\documentclass{article}
\usepackage{graphicx} 
\usepackage[utf8]{inputenc}
\usepackage{amsfonts,amssymb} 
\usepackage{epstopdf}
\usepackage{inputenc}
\usepackage{amsthm}
\usepackage{enumitem}
\usepackage{cite}
\usepackage{mathtools}
\usepackage{graphicx}
\usepackage{authblk}
\usepackage{aligned-overset}
\usepackage{stackengine} 
\usepackage{hyperref}
\hypersetup{
hidelinks,
colorlinks=true,
linkcolor=red,
citecolor=blue,
urlcolor=black
}

\usepackage{booktabs}
\usepackage{siunitx}
\usepackage{caption}
\usepackage{lipsum}
\usepackage[stdsubgroups,nocfg]{nomencl}
\usepackage{setspace}
    \makenomenclature
\hypersetup{
hidelinks,
colorlinks=true,
linkcolor=red,
citecolor=blue,
urlcolor=black
}

\makeatletter
\def\namedlabel#1#2{\begingroup
    #2%
    \def\@currentlabel{#2}%
    \phantomsection\label{#1}\endgroup
}
\makeatother

\begin{document}
\newtheorem{theorem}{Theorem}[section]
\newtheorem{definition}[theorem]{Definition}
\newtheorem{remark}[theorem]{Remark}
\newtheorem{proposition}[theorem]{Proposition}
\newtheorem{lemma}[theorem]{Lemma}
\newtheorem{assumption}{Assumption}
\newtheorem{corollary}[theorem]{Corollary}

\newcommand{\bbK}{{\mathbb{K}}}
\newcommand{\bbF}{{\mathbb{F}}}
\newcommand{\bbR}{{\mathbb{R}}}
\newcommand{\bbQ}{{\mathbb{Q}}}
\newcommand{\bbC}{{\mathbb{C}}}
\newcommand{\bbZ}{{\mathbb{Z}}}
\newcommand{\bbN}{{\mathbb{N}}}
\newcommand{\bbD}{{\mathbb{D}}}
\newcommand{\bbT}{{\mathbb{T}}}
\def\bbA{{\mathbb{A}}}
\newcommand{\bbS}{{\mathbb{S}}}

\title{Continuity of Lyapunov exponents for $\mathcal{C}^r$ one-dimensional maps}
\author{Alexandre Delplanque, Hengyi Li}

\date{October 2025}

\maketitle

\begin{abstract}
    We prove the entropic continuity of Lyapunov exponent for $\mathcal{C}^r$ maps of the interval or of the circle with large entropy for $r>1$, without making any assumptions on the set of critical points. A consequence is the upper semi-continuity of entropy at ergodic measures with large entropy. Another consequence is the uniform integrability of the geometric potential at ergodic measures with large entropy.
\end{abstract}

\section{Introduction}

Lyapunov exponents and entropy are two key invariants in smooth ergodic theory, the first measures the asymptotic rate of divergence of nearby orbits and the second measures the exponential growth of the number of distinct orbits. 
The goal of this paper is to understand the relations between the continuity of entropy and Lyapunov exponent for $\mathcal{C}^r$ one-dimensional maps. In this paper, we show that given a convergent sequence of measures $\nu_k\rightarrow \mu$ whose entropy converges to the topological entropy, its Lyapunov exponent also converges, i.e. $\lim_{k\rightarrow \infty} \lambda(\nu_k)=\lambda(\mu)$. Moreover, we obtain the uniform integrability of the geometric potential $\log |f'|$.\\

Similar results have been obtained by the second author for smooth interval maps with only non-flat critical points \cite{L}, as well as for surface diffeomorphisms by Buzzi-Crovisier-Sarig \cite{BCS} and Burguet \cite{B}.
We expect similar results to hold in any dimension, for arbitrary smooth maps, regardless of the existence or the flatness of their critical sets.\\

\subsection{Statements of the main results}

Let $I$ be the interval $[0,1]$ or the circle $\mathbb{S}^1$.
A map $g:I\rightarrow I$ is $\mathcal{C}^r$ for $r>1$, $r\in \mathbb{R}$, if $g$ is $\mathcal{C}^{\lfloor r \rfloor}$ and $g^{(\lfloor r \rfloor)}$ is $(r-\lfloor r \rfloor)$-H{\"o}lder, where $\lfloor r\rfloor:=\max\{ n\in \mathbb{N}: n\leq r \}$.
For all $g: I\rightarrow I$, $\mathcal{C}^r$-map with $r> 1$, denote by $\mathbb{P}(g)$ the set of $g$-invariant Borel probability measures on $I$; denote by $\mathbb{P}_{\rm{erg}}(g)$ the set of ergodic measures in $\mathbb{P}(g)$. In this paper, we will only consider Borel measures. Furthermore, when we refer to an ergodic measure, we always assume that it is an invariant probability measure.\\

For each $x\in I $, denote the Lyapunov exponent of $x$ by
$$
\lambda_g(x):=\limsup_{n\rightarrow \infty}\frac{1}{n}\sum_{i=0}^{n-1}\log|g'(g^i x)|.
$$
For each $\nu\in \mathbb{P}(g)$, denote the Lyapunov exponent of $\nu$ by
$$
\lambda_g(\nu):=\int \lambda_g(x)d\nu.
$$
By Birkhoff's ergodic theorem, we see that $\lambda_g(\nu)=\int \log|g'|d\nu$. Indeed, when $\log|g'|\notin \mathbb{L}^1(\nu)$, we may still apply Birkhoff's ergodic theorem to $\max\{ -M,\log|g'| \}$ and then take limits in $M$.
Denote by $R(g)$ the spectral radius of $g$:
$$
R(g):=\lim_{n\rightarrow \infty}\frac{1}{n}\log^+ || (g^n)' ||_{\infty},
$$
where $\log^+ = \max \{0, \log\}$ and $||-||_{\infty}$ is the supremum of a map over $I$, not the essential norm.\\

Denote by $h_{\rm{top}}(g)$ the topological entropy of $g$. For each $\nu\in \mathbb{P}(g)$, denote by $h_g(\nu)$ the metric entropy of $\nu$. By the variational principle,
$$
\sup_{\nu\in \mathbb{P}_{\rm{erg}}(g)} h_g(\nu) =h_{\rm{top}}(g).
$$
When $h_g(\nu)=h_{\rm{top}}(g)$, we say that $\nu$ is a measure of maximal entropy of $g$, which we always abbreviate as an m.m.e.
We say that a sequence of $\mathcal{C}^r$-maps $(f_k)$ converges $\mathcal{C}^r$-weakly, which we denote by $f_k \to f$, $\mathcal{C}^r$-weakly, if $(f_k)$ converges to $f$ in the $\mathcal{C}^1$-topology and if the sequence $(f_k)$ is bounded in the $\mathcal{C}^r$-norm.

\begin{theorem}\label{th:main1}
Let $f : I \to I$ be a $\mathcal{C}^r$ map of positive topological entropy, with $r > 1$.
Let $f_k\rightarrow f$, $\mathcal{C}^r$-weakly. Let $(\nu_k)_{k\in \mathbb{N}}$ be a sequence of $f_k$-ergodic measures, converging to some measure $\mu$ in the weak-$\ast$ topology. Then for all $\alpha>\frac{R(f)}{r}$, there exist two $f$-invariant probability measures $\mu_1$, $\mu_0$, and $\beta\in [0,1]$, with $\mu = \beta \mu_1 + (1-\beta) \mu_0$ and 
\begin{itemize}
    \item[\namedlabel{itm:main11}{1)}] $\mu_1$ is hyperbolic, i.e. $\lambda_f(x)>0$ for $\mu_1$-a.e. $x\in I$;
    \item[\namedlabel{itm:main12}{2)}] $\limsup\limits_{k \to +\infty} h_{f_k}(\nu_k) \leq \beta\cdot h_f(\mu_1)+(1-\beta)\cdot\alpha$;
    \item[\namedlabel{itm:main13}{3)}] $\limsup_{k\rightarrow \infty} \lambda_{f_k}(\nu_k)\geq \beta\lambda_f(\mu_1)$.
\end{itemize}
\end{theorem}

\begin{remark}
    The statement is different from \cite{BCS}, but similar to \cite{B}. Furthermore we could only have an inequality rather than an equality in item \ref{itm:main13}. One reason is that the existence of a critical set breaks the continuity of the geometric potential $\log|f'|$, which means that the limit argument used in \cite{BCS} cannot be adopted directly in our case. A more detailed treatment is given in the proof of item \ref{itm:main21'} of Theorem \ref{th:main2}.
\end{remark}

At the end of section \ref{section:proof-th}, we show that one can choose $r=\infty$.

\begin{proposition}
\label{prop:Cinf-case}
If $r = \infty$, then the above statement is true for $\alpha = 0$.
\end{proposition}

We give several corollaries. The first one is that we recover the upper semi-continuity of the topological entropy at $\mathcal{C}^r$ one-dimensional maps of high enough topological entropy, as proven by Burguet in \cite{BurguetJumpEntropy}.

\begin{corollary}[Continuity of the topological entropy]
Let $f_k \to f$, $\mathcal{C}^r$-weakly, with $r > 1$, and assume that $h_{\rm{top}}(f) \geq \frac{R(f)}{r}$, then
$$
\lim\limits_{k \to +\infty} h_{\rm{top}}(f_k) = h_{\rm{top}}(f).
$$
\end{corollary}

\begin{proof}
The lower semi-continuity has been proved by Misiurewicz in \cite{Misiurewicz_LSC_entropy}, we only show the upper semi-continuity to conclude.
Let $(\nu_k)$ be a sequence of $f_k$-invariant ergodic measures such that
\begin{align*}
\limsup\limits_{k \to +\infty} h_{\rm{top}}(f_k)
= \limsup\limits_{k \to +\infty} h_{f_k}(\nu_k).
\end{align*}
Therefore for any $\alpha > R(f)/r$, Theorem \ref{th:main1} gives $\mu_1$ and $\beta$ such that
$$
\limsup\limits_{k \to +\infty} h_{\rm{top}}(f_k) \leq \beta h_f(\mu_1) + (1-\beta) \alpha
\leq \beta h_{\rm{top}}(f) + (1-\beta) \alpha.
$$
By letting $\alpha \to h_{\rm{top}}(f)$ we have
$$
\limsup\limits_{k \to +\infty} h_{\rm{top}}(f_k) \leq h_{\rm{top}}(f).
$$
\end{proof}

\begin{definition}
Let $(\nu_k)_{k\in \mathbb{N}^{\ast}}$ be a sequence of $f_k$-invariant measures on $I$. We say that it \emph{approximates $f$ in entropy,} if 
\begin{itemize}
\item[1.] each $\nu_k$ is ergodic for $f_k$;
\item[2.] the sequence converges in the weak-$*$ topology: $\nu_k \xrightarrow[]{\ast} \mu$ for some $\mu$;
\item[3.] the entropy converges to the topological entropy: $\lim\limits_{k\rightarrow \infty }h_{f_k}(\nu_k)=h_{\rm top}(f).$
\end{itemize}
\end{definition}

We show how this yields the entropic continuity of the Lyapunov exponent.
We also recover the existence of measures of maximal entropy for maps with large topological entropy. This was proven by Burguet in \cite{BurguetExistMME}, as an improvement on Buzzi and Ruette's result in \cite{BuzziRuette}.

\begin{corollary}[Existence of measures of maximal entropy and entropic continuity of the Lyapunov exponent]
\label{cor:Cont-expo}
Assume that $h_{\rm{top}}(f)>\frac{R(f)}{r}$. Let $f_k \to f$, $\mathcal{C}^r$-weakly, and let $(\nu_k)_{k\in \mathbb{N}}$ be a sequence of measures approaching $f$ in entropy.
Then the limit $\mu$ of $(\nu_k)$ is a measure of maximal entropy for $f$ and
$$
\lim_{k\rightarrow \infty} \lambda(\nu_k)=\lambda(\mu).
$$
\end{corollary}
\begin{proof}
We show that $\lambda(\mu)$ is the only accumulation point of the bounded sequence $(\lambda(\nu_k))_k$.
This sequence is positive by Ruelle's inequality \cite{Ruelle1978} and bounded from above by $\sup\limits_{k} \log ||f_{k}'||_{\infty}$, which is finite since $(f_k)_k$ is bounded for the $\mathcal{C}^1$-norm. 
Let $(n_k)_k$ be such that $(\lambda(\nu_{n_k}))_k$ converges to some $\lambda$.
We then apply Theorem \ref{th:main1} to $(\nu_{n_k})$.
Since $(\nu_k)$ approximates $f$ in entropy, it is still true for $(\nu_{n_k})$, hence item \ref{itm:main12} from Theorem \ref{th:main1} gives
\begin{align*}
h_{\rm {top}}(f)=&\lim_{k\rightarrow \infty} h_{f_{n_k}}(\nu_{n_k})\\
&\leq \beta h_f(\mu_1)+(1-\beta)\cdot \alpha\\
&\leq \beta\cdot h_{\rm {top}}(f)+(1-\beta)\cdot \alpha.
\end{align*}
So $(1-\beta)h_{\rm{top}}(f)\leq (1-\beta)\alpha$. But we may choose $\alpha<h_{\rm{top}}(f)$, so $\beta=1$, and $\mu=\mu_1$ is a measure of maximal entropy.
Furthermore, item \ref{itm:main13} from Theorem \ref{th:main1} gives
$$
\lim\limits_{k \to \infty} \lambda(\nu_{n_k}) \geq \lambda(\mu).
$$
Recall that, as the infimum of a sequence of continuous maps, the map $\mu \mapsto \lambda(\mu)$ is upper semi-continuous.
Therefore,
$
\lambda = \lim\limits_{k \to +\infty} \lambda(\nu_{n_k}) \leq \lambda(\mu)
$.
\end{proof}

By the same proof as Corollary 3 from \cite{B}, we obtain:

\begin{corollary}[Upper semi-continuity of the entropy at ergodic measures with large
entropy]
Let $f: I \to I$ be a $\mathcal{C}^r$ map with $r > 1$.
Let $\mu$ be an ergodic measure with $h_f(\mu) \geq \frac{R(f)}{r}$.
Then
$$
\limsup\limits_{\substack{
(g, \nu) \to (f, \mu)\\
\text{s.t. } \nu \text{ is } g-invariant
}} h_g(\nu) \leq h_f(\mu),
$$
where the convergence happens $\mathcal{C}^r$-weakly for $g$ and in weak-$*$ topology for $\nu$.
\end{corollary}

We can obtain the uniform integrability of the geometric potential at measures with large entropy. 

\begin{theorem}[Uniform integrability]
\label{th:main3}
Let $f : I \to I$ be a $\mathcal{C}^r$ map such that $h_{\text{top}}(f) > \frac{R(f)}{r}$.
For any $\varepsilon > 0$, there exist $\delta > 0$, $\eta>0$ and a $\mathcal{C}^r$-weak neighborhood $\mathcal{U}$ of $f$ such that for any $g \in \mathcal{U}$ and any $g$-ergodic measure $\nu$ with $h_g(\nu) \geq h_{top}(g) - \delta$, we have
$$
\int_{\{x: |g'(x)|<\eta \}} -\log|g'|d\nu<\varepsilon.
$$
\end{theorem}

This gives the following corollary if we do not consider perturbations of $f$.

\begin{corollary}
Let $f : I \to I$ be a $\mathcal{C}^r$ map such that $h_{\text{top}}(f) > \frac{R(f)}{r}$.
For any $\varepsilon > 0$, there exists $\delta > 0$ and $\eta>0$ such that for any $f$-ergodic measure $\mu$ with $h_f(\mu) \geq h_{top}(f) - \delta$, we have
$$
\int_{\{x: |f'(x)|<\eta \}} -\log|f'|d\mu<\varepsilon.
$$
\end{corollary}

\subsection{Consequences for SRB measures}

We state some consequences of the previous statements regarding the stability of SRB measures for $\mathcal{C}^{\infty}$ maps.
We recall the following entropic characterization of SRB measures for one-dimensional maps:

\begin{theorem}[Theorem VII.1.1 from \cite{SETE}]
\label{th:entropy-form}
Let $f : I \to I$ be a $\mathcal{C}^r$ map where $r > 1$.
Let $\mu$ be an $f$-invariant hyperbolic Borel probability measure.
Then $\mu$ is absolutely continuous with respect to the Lebesgue measure on $I$ if and only if it satisfies the entropy formula
$$
h_{\mu}(f) = \int_I \lambda_f(x) \: d\mu.
$$
\end{theorem}

\begin{corollary}
\label{cor:compo-SRB}
Let $(f_k : I \to I)$ be a sequence of $\mathcal{C}^{\infty}$ maps that converges to $f$ in the $\mathcal{C}^{\infty}$ topology.
Assume that each $f_k$ admits an ergodic SRB measure $\mu_k$ whose entropy is away from 0 uniformly in $k$.
Then any accumulation point of $(\mu_k)$ has an SRB component.
\end{corollary}

\begin{proof}
For simplicity of notations, assume that $(\mu_k)$ converges to $\mu$.
From Proposition \ref{prop:Cinf-case}, we write $\mu = \beta \mu_1 + (1-\beta) \mu_0$.
We show that $\mu_1$ is SRB.
From Theorem \ref{th:main1}, the measure $\mu_1$ is hyperbolic and we have
$$
\beta h_f(\mu_1) \geq \limsup\limits_{k \to +\infty} h_{f_k}(\mu_k)
= \limsup\limits_{k \to +\infty} \lambda_{f_k}(\mu_k)
\geq \beta \lambda_f(\mu_1).
$$
The fact that $(h_{f_k}(\mu_k))_k$ is uniformly bounded away from zero then gives $\beta > 0$, hence $\mu_1$ is a nontrivial component of $\mu$ and Theorem \ref{th:entropy-form} concludes.
\end{proof}

We may also state the previous corollary as follows :

\begin{corollary}
For any fixed $h > 0$, the following set is closed in $\mathcal{C}^{\infty}(I,I)$ :
$$
\{ f \in \mathcal{C}^{\infty}(I,I) : f \text{ admits an SRB measure } \mu \text{ with } h_f(\mu) \geq h \}.
$$
\end{corollary}

In the case where the SRB measures approach $f$ in entropy, we show that their limit is SRB.

\begin{corollary}
\label{cor:mme-SRB}
Suppose that $(f_k)$ converges to $f$ in the $\mathcal{C}^{\infty}$ topology and that each $f_k$ admits an ergodic SRB measure $\mu_k$.
Suppose that $h_{top}(f) > 0$.
If $\mu_k \to \mu$ and $h_{f_k}(\mu_k) \to h_{top}(f)$, then $\mu$ is SRB.
\end{corollary}

\begin{proof}
From Proposition \ref{prop:Cinf-case}, we have
\begin{align*}
h(\mu) &\geq \beta h(\mu_1)\geq \limsup\limits_{k \to +\infty} h_{f_k}(\mu_k)= \limsup\limits_{k \to +\infty} \lambda_{f_k}(\mu_k)
\underset{\text{Corollary } \ref{cor:Cont-expo}}{=} \lambda_f(\mu).
\end{align*}
To apply Theorem \ref{th:entropy-form}, it remains to show that $\mu$ is hyperbolic.
Since $\mu$ is an m.m.e, every ergodic component of $\mu$ is also an m.m.e, so they are all hyperbolic, which implies that $\mu$ is hyperbolic.
\end{proof}

\subsection{Ideas of the proof}

The proof is based on the works of Burguet \cite{B} and Buzzi-Crovisier-Sarig \cite{BCS} for surface diffeomorphisms  and of the first author for interval maps \cite{A}. Adopting the ideas in \cite{B} and \cite{BCS}, we group the iterates into blocks of two types, namely the hyperbolic ones (as in \cite{B}) and the neutral ones (as in \cite{BCS}). Within each hyperbolic block, we expect to see continuity of the Lyapunov exponent. On neutral blocks, the continuity of the Lyapunov exponent is less controlled, but we will prove that its contribution to entropy is uniformly bounded. This is done by adapting the reparametrization lemma of Burguet from \cite{B}, as done by the first author in \cite{A}. In the end, we get an entropy bound in terms of the portion of hyperbolic blocks, as formulated in item \ref{itm:main12} in Theorem \ref{th:main1}. \\ 

A difference between our approach and that of Burguet lies in the treatment of continuity of Lyapunov exponents along hyperbolic blocks. For Burguet, since he works on surface diffeomorphisms, this continuity is guaranteed once he passes to the projective tangent bundle. For us, with the map admitting critical points, we need to use the fact that iterates are uniformly away from the critical set along hyperbolic blocks.\\

Another difference is that we prove the uniform integrability of the geometric potential $\log|f'|$. This is done by relating this property to the continuity of the Lyapunov exponent at measures of high entropy.\\

In comparison to previous works, the current one overcomes the difficulties brought by the existence of a critical set. In \cite{L}, the assumption of non-flatness has to be imposed for a geometric treatment. Here, with the reparametrization technique established in \cite{A}, we could investigate the behavior of the dynamics near the critical set in a more sophisticated manner.

\subsection{Organization of the paper}

In section \ref{sec:PC}, we introduce the notion of hyperbolic and neutral blocks for general interval systems. We then split the entropy with respect to these blocks.\\

In section \ref{sc:ENC}, we study the entropy contribution from neutral blocks. With the reparametrization lemma, we show that this contribution is uniformly bounded.\\

In section \ref{sec:main-prop}, we show the continuity of the entropy and of the Lyapunov exponent along hyperbolic blocks, which gives a part of item \ref{itm:main13} of Theorem \ref{th:main1} and is the first step to prove item \ref{itm:main12} of Theorem \ref{th:main1}.\\

In section \ref{sc:Proof}, we conclude the proof of Theorem \ref{th:main1}. In section \ref{sec:expo-neutral}, we prove item \ref{itm:main13} of Theorem \ref{th:main1}. In section \ref{ssec:proof}, we prove Theorem \ref{th:main1} for an iterate of $f$. 
We then show in section \ref{section:proof-th} how this implies Theorem \ref{th:main1}.\\

In section \ref{sc:UI}, we treat the uniform integrability as stated in Theorem \ref{th:main3}.

\subsection{Notations}

In sections \ref{sec:PC} and \ref{sc:ENC}, by $f:I\rightarrow I$ we refer to a $\mathcal{C}^r$ map in general. Indeed the results that we prove in these sections are not directly related to continuity. From section \ref{sec:main-prop} to section \ref{sc:UI}, by $f$ we refer to the weak $\mathcal{C}^r$-limit of $(f_k)$, since here we investigate the continuity property and need to keep track of the perturbations.\\

Often we work with integer intervals and subsets of integers. For $a<b\in \mathbb{Z}\cup \{+\infty,-\infty\}$, we let $[\![ a,b [\![:=\{ c\in \mathbb{Z}: a\leq c<b \}$. Given $E\subseteq \mathbb{Z}$, denote its boundary by $\partial E:=\{ a\in E: a-1\notin E \}\cup \{ a\notin E: a-1\in E \}$.

\section{Preliminary construction}
\label{sec:PC}

\subsection{Hyperbolic times}
\label{ssec:hyperbolic-times}

Let $f:I\rightarrow I$ be a $\mathcal{C}^r$ map of the interval or the circle to itself, $r>1$, and let $\nu$ be an ergodic measure for $f$.
For $i\in \mathbb{N}$, denote
$$
\psi_i(x):=\log|(f^i)'(x)|-\frac{1}{r}\sum_{j=0}^{i-1}\log^{+}|f'(f^j x)|.
$$
Assume that $\int \psi_1 d \nu > 0$.
For $\delta > 0$ and $x\in I$, we denote by
$$
E_{\delta}(x)=E(x):=\{ l\in \mathbb{N}: \text{for all $i\in [\![ 0,l[\![,$ $\psi_{l-i}(f^i x)\geq (l-i)\cdot \delta$}\}
$$
the set of $\delta$-hyperbolic times of $x$. 
We state a first lemma related to the lower density of $E_{\delta}(x)$, defined as
$$
\underline{d}(E_{\delta}(x)): = \liminf\limits_{n \to +\infty} \frac{\sharp (E_{\delta}(x) \cap [\![0 ; n [\![)}{n}.
$$

By applying the same argument as in Lemma 3 from \cite{B} but in the inverse space limit (see \cite{SarigNotes}), one can obtain :

\begin{lemma}[Lemma 3 from \cite{B}]
\label{lemma:delta}
If $0  < \delta < \int \psi_1 d \nu$, then $E_{\delta}(x)$ has positive lower density for $\nu$-almost every $x \in I$.
\end{lemma}

\begin{proof}
We define the inverse space limit as
$$
I^f = \{ (x_n)_{n \in \bbZ} \in I^{\bbZ} : f(x_n) = x_{n+1} \}.
$$
We denote by $\sigma$ the shift on $I^f$, so that the dynamical system $(I^f, \sigma)$ is semi-conjugated to $(I,f)$, through the semi-conjugacy $\pi:(x_n)\in I^f\mapsto x_0\in I$.
Furthermore, there exists a unique $\sigma$-ergodic measure $\widetilde{\nu}$ such that $\pi_* \widetilde{\nu} = \nu$.
Then, we let
$$
H_{\delta} = \{ \widetilde{x} \in I^f : \forall l > 0, \psi_l(\pi\circ\sigma^{-l} \widetilde{x}) \geq l \delta \}.
$$
Let $A_{\delta} = I^f \backslash H_{\delta}$.
Let $\widetilde{\phi_{\delta}} = \delta - \psi_1\circ\pi\circ\sigma^{-1} : I^f \to \bbR$.
Therefore, we have
$$
A_{\delta} = \{ \widetilde{x} \in I^f :  \sup\limits_{l > 0} \frac{1}{l} \sum\limits_{k =0}^{l-1} \widetilde{\phi_{\delta}}(\sigma^{-k}(\widetilde{x})) > 0 \}.
$$
By using the ergodic maximal inequality \cite{P}, we have
$
\int_{A_{\delta}} \widetilde{\phi_{\delta}} d \widetilde{\nu} \geq 0
$.
Thus
$
\int_{H_{\delta}} \widetilde{\phi_{\delta}} d \widetilde{\nu}
\leq \int \widetilde{\phi_{\delta}} d \widetilde{\nu}
= \delta - \int \psi_1 d \nu < 0
$.
Therefore, $\widetilde{\nu}(H_{\delta}) > 0$.
By letting $\widetilde{E}(\widetilde{x}) = \{n \in \bbN : \sigma^n \widetilde{x} \in H_{\delta} \}$, the fact that $\widetilde{\nu}$ is ergodic gives that for $\widetilde{\nu}$-a.e. $\widetilde{x}$, we have
$$
\frac{1}{n} \sharp \left ( \widetilde{E}(\widetilde{x}) \cap [\![0 ; n [\![ \right ) \underset{n \to +\infty}{\longrightarrow} \widetilde{\nu}(H_{\delta}).
$$
Then, the image by $\pi$ of this full $\widetilde{\nu}$-measure set is of full $\nu$-measure, and one can see that $\widetilde{E}(\widetilde{x}) \subset E(\pi(\widetilde{x}))$, therefore, for $\nu$-a.e. $x$, we have
$$
\liminf\limits_{n \to +\infty} \frac{1}{n} \sharp \left ( E(x) \cap [\![0 ; n [\![ \right ) \geq \widetilde{\nu}(H_{\delta}). 
$$
In particular, $E(x)$ is infinite for $\nu$-a.e $x$.
\end{proof}

\begin{remark}
\label{rmk:density-E}
The proof actually gives the following lower bound on the lower density of $E_{\delta}(x)$, which is itself a lower bound on the $\beta$ that appears in the statements in the introduction: for $\nu$-a.e. $x \in I$, we have
$$
\underline{d}(E_{\delta}(x)) \geq \frac{\int \psi_1 d \nu - \delta}{\log || f' ||_{\infty}- \delta}.
$$
\end{remark}

Fix $\delta > 0$ such that $\delta < \int \psi_1 d \nu$.
For each $n\geq 0$, let $E_n(x):=E(x)\cap [\![0,n[\![$, $m_n(x):=\max E_n(x)$ and $m_n^+(x):=\min (E(x)\setminus [\![0,n[\![)$. 
We show that these quantities are asymptotically equivalent to $n$.

\begin{lemma}\label{le:bn}
For $\nu$-a.e. $x\in I$, we have
$$
\lim_{n\rightarrow \infty} \frac{m_n(x)}{n}=1
\;\;\; \text{and}\;\;\;
\lim_{n\rightarrow \infty} \frac{m_n^+(x)}{n}=1.
$$
\end{lemma}

\newpage
To prove this lemma, we use Lemma 2.9 from \cite{L}:

\begin{lemma}\label{le:erg}
 Let $(X,\mathcal{B},f,\nu)$ be an ergodic system. Let $\Psi:X\rightarrow [-\infty,\infty)$ be a $\mathbb{L}^1(\nu)$-function and let $\Delta > 0$ such that $\int \Psi d \nu < \Delta$.
 Assume that for $\nu$-a.e. $x \in X$, there exist two sequences of integers $(a_n)_{n \in \bbN}$ and $(b_n)_{n \in \bbN}$ such that 
 \begin{itemize}
 \item[1)] $a_n\leq n\leq b_n$, $a_n \uparrow \infty$;
 \item[2)] for all $n$, $\sum_{i=a_n}^{b_n-1}\Psi(f^i x)\geq\Delta\cdot(b_n-a_n)$.
 \end{itemize}
 Then for $\nu$-a.e. $x\in X$, $b_n-a_n=o(n)$. 
\end{lemma}

Now we proceed to prove Lemma \ref{le:bn}

\begin{proof}[Proof of Lemma \ref{le:bn}]
It suffices to check that there exists a
$\Psi$ such that $a_{n-1} = m_n$ and $b_{n-1} = m_n^+ - 1$ satisfy all the conditions given in Lemma \ref{le:erg}.
Let $\Psi(x):=-\psi_1(x)+\log||f'||_{\infty}$ and $\Delta:=-\delta+\log||f'||_{\infty}$. From the assumption $\int \psi_1 d \nu > 0$, we have $\Psi\in \mathbb{L}^1(\nu)$ and $\int \Psi d\nu<\Delta$. The fact that $m_n \uparrow \infty$ is already proved in Lemma \ref{lemma:delta}. So it suffices to check that the second condition is satisfied.\\

For $n \geq 1$, we may have $m_n^+-1 \in E(x)$, in which case we have $m_n^+ = n$ and $m_n = n-1$ and the inequality is immediate. Then, if $m_{n}^{+}-1\notin E(x)$, there exists $l>0$ such that
$$
\begin{aligned}
\sum_{i=m_n^+-1-l}^{m_n^+-2}\Psi(f^i x)&=-\psi_l (f^{m_n^+-1-l}x)+\log||f'||_{\infty}
\cdot l\\
&\geq (\log||f'||_{\infty}-\delta)\cdot l.
\end{aligned}
$$
Since $]\!]m_n,m_n^+[\![\cap E(x)=\emptyset$, we may repeat this process inductively until $m_n^+-1-l\leq m_n$. Now that $m_n\in E(x)$, we have
$$
\begin{aligned}
\sum_{i=m_n^+-l-1}^{m_n-1}\Psi(f^i x)&=-\psi_{m_n-(m_n^+-1-l)}(f^{m_n^+-1-l}x)+\log||f'||_{\infty}\cdot (m_n-(m_n^+-1-l))\\
&\leq (\log||f'||_{\infty}-\delta)\cdot (m_n-(m_n^+-1-l)).
\end{aligned}
$$
The two inequalities combined give the desired upper-bound.
\end{proof}

By Egorov theorem, there exists a compact subset $\mathtt{F}_{\nu}\subseteq I$ such that $\nu(\mathtt{F}_{\nu})$ is arbitrarily close to $1$ and such that the convergences from Lemma \ref{le:bn} are uniform on $\mathtt{F}_{\nu}$.
Denote by $\zeta(-):=\frac{\nu(-\cap\mathtt{F}_{\nu})}{\nu(\mathtt{F}_{\nu})}$ the normalized restriction of $\nu$ to $\mathtt{F}_{\nu}$.
Taking a restriction to $\mathtt{F}_{\nu}$ will be useful for the bounds in the Reparametrization Lemma, namely Lemmas \ref{lemma:repabound} and \ref{le:combinatorics}, and to estimate the exponent along the neutral blocks in Proposition \ref{prop:expo-neutral}.
\\

Given $M\in \mathbb{N}$, $x\in I$ let
$$
E^M_n(x):=\bigcup_{a,b\in E(x)\cap [\![0,n]\!], |b-a|\leq M} [\![a,b[\![,
$$
and we call the integer connected components of $E^M_n(x)$ the $M$-hyperbolic blocks of $x$.\\

We let
$$
\xi^{M}_{\nu,n}:=\int \frac{1}{n}\sum_{i\in E^M_n(x)}\delta_{f^i x}d\nu(x) \;\;\; \text{and} \;\;\; \beta_{\nu,n}^M = \frac{1}{n} \int \sharp E_n^M(x) d \nu(x).
$$
We may take a subsequence and assume that $(\xi_{\nu,n}^M)_n$ converges in the weak-$*$ topology.
By Cantor's diagonal argument, we may also assume that this subsequence does not depend on $M$.
We will also consider that $n$ always belongs to this subsequence, so that we write
$$
\xi^M_{\nu}:=\lim_{n\rightarrow \infty}\xi^{M}_{\nu,n}.
$$
We call $\xi^M_{\nu}$ the hyperbolic component of $\nu$. Denote by $\overline{\xi^M_{\nu}}$ the normalization of the measure $\xi^M_{\nu}$. Write as $\beta^M_{\nu}$ the coefficient $\xi^M_{\nu}=\beta^M_{\nu}\cdot \overline{\xi^M_{\nu}}$.
So that $\beta_{\nu}^M = \lim\limits_{n \to +\infty} \beta_{\nu,n}^M$.
For each $\xi^M_{\nu,n}$, denote by $\eta^M_{\nu,n}:=\nu-\xi^M_{\nu,n}$ its complement.
We denote by $\eta^M_{\nu}:=\lim_{n\rightarrow \infty} \eta^M_{\nu,n}$ the neutral component of $\nu$.

\begin{lemma}[Lemma 4.1 from \cite{A}]
\label{lem:EnM}
For $M,n \in \bbN$ and $x \in I$, the set $E_n^{M}(x)$ satisfies the following properties:
\begin{itemize}
\item[\namedlabel{itm:EnM1}{i)}] $\partial E_n^{M}(x) \subset E(x)$;
\item[\namedlabel{itm:EnM2}{ii)}] $\limsup\limits_{n \to + \infty} d_n(\partial E_n^{M}(x)) \leq \frac{2}{M}$, where $
d_n S:=\frac{\sharp(S\cap [\![0,n[\![)}{n}$ for $S \subset \bbN$;
\item[\namedlabel{itm:EnM3}{iii)}] $M \frac{\sharp \partial E_n^{M}(x)}{2} \leq n + M.$
\end{itemize}
\end{lemma}

We also have that hyperbolic blocks are not too close to critical points.

\begin{lemma}[Lemma 4.2 from \cite{A}]
\label{lemma:prop-hb}
Let $x \in I$, $n, M \in \bbN$, and $[\![ l_1 ; l_2  [\![ \subset E_n^M(x)$. We have
$$
\log |(f^{l_2-l_1})'(f^{l_1} x)| \geq -M \log || f' ||_{\infty}.
$$
\end{lemma}

We introduce some other notations.
Given $M,n\in \mathbb{N}$, and $m,q\in \mathbb{N}^{\ast}$,
write $E^M_n(x)=\bigsqcup_{s} [\![a_s,b_s[\![$ where $[\![a_s,b_s[\![$ are $M$-hyperbolic blocks. Let
$$
E^{M,m}_n(x):=\bigsqcup_{s\in \mathbb{N}} [\![a_s,c_s[\![,
$$
where $c_s:=\max\{ l\in [\![a_s,b_s+1-m]\!]: l\in E(x)\}$.\\

We also define $E_n^{M,m,q}(x) = E_n^{M,\max \{m,q \}}(x)$.
Then, let
$$
\xi^{M,m,q}_{\nu,n}:=\int \frac{1}{n}\sum_{i\in E^{M,m,q}_n(x)}\delta_{f^i x}d\nu(x) \;\;\; \text{and} \;\;\; \beta_{\nu,n}^{M,m,q} = \frac{1}{n} \int \sharp E_n^{M,m,q}(x) d \nu(x).
$$

We remark that $E^{M,1,1}_n(x)=E^{M}_n(x)$, and that $\xi^{M,1,1}_{\nu,n}=\xi^M_{\nu,n}$.\\

The motivation for these definitions is the following. We will estimate the entropy of $f$ and the Lyapunov exponent of $f^q$.
In doing so, we will require control over $|(f^q)'|$, and when we compute the entropy of $f$ for some partition, we will require some control over $|f' \circ f^m|$ for large $m$.
From Lemma \ref{lemma:prop-hb}, we control $|f'(f^k x)|$ when $k \in E_n^M(x)$.
Notice that $|(f^q)'(f^l x)|$ is a finite sum of $|f'(f^{l+k} x)|$ with $k < q$, thus if $l \in E_n^{M,m,q}(x)$ then $l+k \in E_n^M(x)$ and we control $|(f^q)'(f^l x)|$.
Then $|f' \circ f^m (f^l x)|$ is also controlled since $l+m \in E_n^M(x)$.
Therefore, at times of $E_n^{M,m,q}(x)$ of the orbit of $x$, we control $|f' \circ f^m|$ and $|(f^q)'|$.
Also, the choice of $c_s$ gives that the border of $E_n^{M,m,q}(x)$ is still contained in $E(x)$.

\subsection{Entropy splitting of the hyperbolic and neutral component}
\label{ssec:gen-entropy}

For the beginning of this section, we consider a general probability space $(X, \mathcal{A}, \rho)$.
We will use the following notion of measurable partition, which depends on the measure:

\begin{definition}[Measurable partition]
We say that a countable collection $\mathcal{Q}$ of measurable subsets of $X$ is a measurable partition for $\rho$ if
\begin{itemize}
\item[i)] for $Q \neq Q' \in \mathcal{Q}$, $\rho(Q \cap Q') = 0$;
\item[ii)] the union $\bigcup\limits_{Q \in \mathcal{Q}} Q$ has full $\rho$-measure.
\end{itemize}
\end{definition}

Then, for a measurable partition $\mathcal{Q}$, we define
$$
\mathcal{Q}_{\rho} := \{ Q \in \mathcal{Q} : \rho(Q) > 0 \}.
$$
The important remark is that when a measure $\rho$ is not preserved by a dynamical system $f$, and if $\mathcal{Q}$ is a measurable partition for $\rho$, then $f^{-1} \mathcal{Q}$ may not be a measurable partition for $\rho$.
By adapting an argument from Misiurewicz's proof of the variational principal \cite{Misiurewicz1976}, we obtain: 

\begin{lemma}[Lemma 6.2 from \cite{A}]
\label{lemma:Misiurewicz}
Let $\mathcal{Q}$ be a countable measurable partition of a probability space $(X, \rho)$. Let $T : X \to X$ be a measurable transformation, which may not preserve $\rho$. Let $F$ be a finite subset of $\bbN$.
Let $\rho^F = \frac{1}{\sharp F} \sum\limits_{k \in F} T^k_* \rho$.
Assume that for every $i \in F$, $T^{-i} \mathcal{Q}$ is a measurable partition for $\rho$ and that $\mathcal{Q}_{\rho^F}$ is finite.
For $m \in \bbN^*$, if we write $\mathcal{Q}^m=\bigvee\limits_{k=0}^{m-1} T^{-k} \mathcal{Q}$ and $\mathcal{Q}^F = \bigvee\limits_{k \in F} T^{-k} \mathcal{Q}$, then we have
$$
\frac{1}{m} H_{\rho^F} (\mathcal{Q}^m) \geq \frac{1}{\sharp F} H_{\rho}(\mathcal{Q}^F) - m \log (\sharp \mathcal{Q}_{\rho^F}) \frac{\sharp \partial F}{\sharp F}.
$$
\end{lemma}

Let $\iota > 0$ and let $\mathcal{P}$ be a finite partition such that $\nu(\partial \mathcal{P}) = 0$ and 
$$
\lim\limits_{n \to +\infty} \frac{1}{n} H_{\nu}(\mathcal{P}^n) \geq h(\nu) - \iota.
$$

For a measurable partition $\mathcal{F}$ for $\nu$ and $n,M,m \in \bbN^*$, we denote by $\mathcal{F}^{E^{M,m}_n}$ the partition for which two points $x$ and $y$ are in the same atom if $E_n^{M,m}(x) = E_n^{M,m}(y)$ and if $f^l x$ and $f^l y$ are in the same atom of $\mathcal{F}$ for every $l \in E_n^{M,m}(x) = E_n^{M,m}(y)$.\\

For $n,M,m$, let $\mathcal{E}_n^{M,m}$ be the partition of $I $ whose atoms are the sets $\{ x \in I :\ E_n^{M,m}(x) = E \}$ where $E \subset \bbN$ are such that the associated atoms are of positive $\nu$-measure.
To shorten notations, we will denote by $E$ both the atom of $\mathcal{E}_n^{M,m}$ and the set $E_n^{M,m}(x)$ for $x$ in this atom. Then for $E \in \mathcal{E}_n^{M,m}$, we write
$$
\nu_E = \frac{\nu(E \cap \cdot)}{\nu(E)} \; \; \; \text{and} \; \; \; \nu^E = \frac{1}{\sharp E} \sum\limits_{k \in E} f^k_* \nu_E.
$$
The main proposition of this section is the following entropy inequality, which splits the entropy of a measure $\nu$ into the sum of the entropy of the hyperbolic component of $\nu$ and the entropy conditioned to the hyperbolic component of $\nu$, with some error terms.

\begin{proposition}
\label{prop:entropy-ineq}
Let $\mathcal{Q}$ be a countable measurable partition of $(I, \nu)$ such that for any $n,M,m$
\begin{itemize}
\item[i)] $\forall E \in \mathcal{E}_n^{M,m}, \forall i \in E, f^{-i} \mathcal{Q}$ is a partition for $\nu_E$;
\item[ii)] $\forall M,m, \sharp \{ Q \in \mathcal{Q}^m : \exists n, \exists E \in \mathcal{E}_n^{M,m}, \nu^E(Q) > 0 \} < +\infty$;
\item[iii)] $\forall m,M \in \bbN^*, \overline{\xi_{\nu}^{M,m}}(\partial \mathcal{Q}^m)=0$.
\end{itemize}
Then for any $M$ large enough and $m \in \bbN^*$,
\begin{align*}
h(\nu) \leq &\beta_{\nu}^{M,m} \frac{1}{m} H_{\overline{\xi_{\nu}^{M,m}}}(\mathcal{Q}^m) + \limsup\limits_{n \to +\infty} \frac{1}{n} H_{\nu}(\mathcal{P}^n \mid \mathcal{Q}^{E_n^{M,m}}) + \frac{3 \log M}{M} + \iota  \\
&+\limsup\limits_{n \to +\infty} \frac{m}{n}\sum\limits_{E \in \mathcal{E}_n^{M,m}} \nu(E) \log \left ( \sharp\mathcal{Q}_{\nu^E} \right ) \sharp\partial E.  
\end{align*}
\end{proposition}

\begin{proof}
Let us fix $n \in \bbN^*$ and estimate $\frac{1}{n} H_{\nu}(\mathcal{P}^n)$.
We have
\begin{align*}
\frac{1}{n} H_{\nu}(\mathcal{P}^n)
&\leq \frac{1}{n} H_{\nu}(\mathcal{P}^n \mid \mathcal{E}_n^{M,m}) + \frac{1}{n} H_{\nu}(\mathcal{E}_n^{M,m})\\
&\leq \frac{1}{n} H_{\nu}(\mathcal{P}^n \mid \mathcal{Q}^{E_n^{M,m}}) + \frac{1}{n} H_{\nu}(\mathcal{Q}^{E_n^{M,m}} \mid \mathcal{E}_n^{M,m}) + \frac{1}{n} H_{\nu}(\mathcal{E}_n^{M,m}).
\end{align*}
For the third term, the entropy of $\mathcal{E}_n^{M,m}$, we let $F_n^{M,m} = \max\limits_{E \in \mathcal{E}_n^{M,m}} \sharp \partial E$. Since any $E \subset \bbN$ is uniquely determined by $\partial E$, and since $F_n^{M,m} \leq \frac{2n}{M}$, by item \ref{itm:EnM2} from Lemma \ref{lem:EnM}, we have, for $M$ large enough:
$$
\frac{1}{n} H_{\nu}(\mathcal{E}_n^{M,m}) \leq \frac{1}{n} \log \sum\limits_{k=0}^{F_n^{M,m}} \binom{n}{k} \leq \frac{3 \log M}{M}.
$$

Then, we deal with the second term. We have
\begin{align*}
H_{\nu}(\mathcal{Q}^{E_n^{M,m}} \mid \mathcal{E}_n^{M,m})
&=\sum\limits_{E \in \mathcal{E}_n^{M,m}} \nu(E) H_{\nu_E}(\mathcal{Q}^E)\\
\underset{\text{Lemma } \ref{lemma:Misiurewicz}}&{\leq} \sum\limits_{E \in \mathcal{E}_n^{M,m}} \nu(E) \sharp E \left ( \frac{1}{m} H_{\nu^E}(\mathcal{Q}^m) + m \frac{\sharp\partial E}{\sharp E} \log \left (\sharp \mathcal{Q}_{\nu^E} \right )\right ).
\end{align*}
Then, we have the equality
$$
\sum\limits_{E \in \mathcal{E}_n^{M,m}} \nu(E) \sharp E \nu^E = \left ( \int \sharp E_n^{M,m}(x) d \nu \right ) \overline{\xi_{\nu,n}^{M,m}}.
$$
Hence, by concavity of $\rho \mapsto H_{\rho}(\mathcal{Q}^m)$, we have
$$
\sum\limits_{E \in \mathcal{E}_n^{M,m}} \nu(E) \sharp E H_{\nu^E}(\mathcal{Q}^m) \leq \left ( \int \sharp E_n^{M,m}(x) d \nu \right ) H_{\overline{\xi_{\nu,n}^{M,m}}}(\mathcal{Q}^m).
$$
Then, recall that $\beta_{\nu,n}^{M,m} = \frac{1}{n}\int \sharp E_n^{M,m} (x) d \nu(x)$, therefore
\begin{align*}
\frac{1}{n} H_{\nu}(\mathcal{P}^n)
&\leq \frac{1}{n} H_{\nu}(\mathcal{P}^n \mid \mathcal{Q}^{E_n^{M,m}})\\
&+ \frac{1}{m} \beta_{\nu,n}^{M,m} H_{\overline{\xi_{\nu,n}^{M,m}}}(\mathcal{Q}^m)
+ \frac{m}{n} \sum\limits_{E \in \mathcal{E}_n^{M,m}} \nu(E) \sharp E \frac{\sharp\partial E}{\sharp E} \log \left (\sharp \mathcal{Q}_{\nu^E} \right )\\
&+ \frac{3 \log M}{M}.
\end{align*}
Then by hypothesis $ii)$, the term $H_{\overline{\xi_{\nu,n}^{M,m}}}(\mathcal{Q}^m)$ is given by a sum over the atoms of $\mathcal{Q}^m$ where only finitely many atoms ever appear when $n$ varies.
Hence by hypothesis $iii)$, this sum converges to the entropy of the partition $\mathcal{Q}^m$ for the limit measure $\overline{\xi_{\nu}^{M,m}}$, which concludes.
\end{proof}

Recall from section \ref{ssec:hyperbolic-times} that $\mathtt{F}_{\nu} \subset I$ is a set of positive $\nu$-measure where the convergence $\frac{m_n^+(x)}{n} \underset{n \to +\infty}{\longrightarrow} 1$ is uniform. Then $\zeta$ is the normalized restriction of $\nu$ to $\mathtt{F}_{\nu}$.
We finish this section by relating the entropy of $\nu$ to the entropy of $\zeta$. We first prove a general inequality.

\begin{lemma}\label{le:entropyclose}
For any measurable partitions $\mathcal{P}$ and $\mathcal{Q}$ of $I$ for $\nu$, we have
$$
H_{\nu}(\mathcal{P}\mid \mathcal{Q}) \leq H_{\zeta}(\mathcal{P} \mid \mathcal{Q}) + \nu(I \backslash \mathtt{F}_{\nu}) \log \sharp\mathcal{P} + \log 2.
$$
\end{lemma}

\begin{proof}
Let us denote by $\mathcal{F}$ the partition $\{ \mathtt{F}_{\nu} ; I \backslash \mathtt{F}_{\nu}\}$ and by $\eta$ the normalized restriction of $\nu$ on $I \backslash \mathtt{F}_{\nu}$.
We have the following:
\begin{align*}
H_{\nu}(\mathcal{P}\mid \mathcal{Q})
&= H_{\nu}(\mathcal{P} \vee \mathcal{F} \mid \mathcal{Q}) - H_{\nu}(\mathcal{F} \mid \mathcal{Q} \vee \mathcal{P})\\
&=H_{\nu}(\mathcal{P} \mid \mathcal{Q} \vee \mathcal{F}) - H_{\nu}(\mathcal{F} \mid \mathcal{Q} \vee \mathcal{P}) + H_{\nu}(\mathcal{F} \mid \mathcal{Q})\\
&= \nu(\mathtt{F}_{\nu}) H_{\zeta}(\mathcal{P} \mid \mathcal{Q})  + \nu(I \backslash \mathtt{F}_{\nu}) H_{\eta}(\mathcal{P}\mid \mathcal{Q}) - H_{\nu}(\mathcal{F}\mid \mathcal{Q} \vee \mathcal{P}) + H_{\nu}(\mathcal{F} \mid \mathcal{Q})\\
&\leq H_{\zeta}(\mathcal{P} \mid \mathcal{Q}) + \nu(I \backslash \mathtt{F}_{\nu}) \log \sharp\mathcal{P} + \log 2.
\end{align*}
\end{proof}

As a consequence, we have the following statement.

\begin{proposition}
\label{prop:entropy-final}
If a partition $\mathcal{Q}$ satisfies the hypotheses of Proposition \ref{prop:entropy-ineq}, then for any $M$ large enough and $m \in \bbN^*$, we have
\begin{align*}
h(\nu) \leq &\beta_{\nu}^{M,m} \frac{1}{m} H_{\overline{\xi_{\nu}^{M,m}}}(\mathcal{Q}^m) + \limsup\limits_{n \to +\infty} \frac{1}{n} H_{\zeta}(\mathcal{P}^n \mid \mathcal{Q}^{E_n^{M,m}})\\
&+ \nu(I \backslash \mathtt{F}_{\nu}) \log \sharp \mathcal{P} + \frac{3 \log M}{M} + \iota  \\
&+\limsup\limits_{n \to +\infty} \frac{m}{n}\sum\limits_{E \in \mathcal{E}_n^{M,m}} \nu(E) \log \left ( \sharp\mathcal{Q}_{\nu^E} \right ) \sharp\partial E.
\end{align*}
\end{proposition}

\subsection{Defining the adapted partition}
\label{ssec:def-partition}

In this section, we define a partition satisfying the hypotheses of Proposition \ref{prop:entropy-ineq} and such that we control the distortion on each atom.\\

Given any $q\in \mathbb{N}$ and $a\in ]-\frac{1}{q},0[$, we define 
$$
I_q:=\{ [\frac{i}{q}+a,\frac{i+1}{q}+a[: i\in \mathbb{Z} \}.
$$
Define the partition
$$
\mathcal{Q}_{q}:=\{ (\log|f'|)^{-1}(J): J\in I_{q} \}.
$$
We also require the atoms to be small enough. Hence, for $\varepsilon > 0$, we consider a partition $\mathcal{I}(\varepsilon)$ whose atoms are interval of diameter smaller than $\varepsilon$.
Finally, let
$$
\mathcal{R}_{q, \varepsilon} = \mathcal{Q}_q \vee \mathcal{I}(\varepsilon).
$$
We choose $a$ and $\mathcal{I}(\varepsilon)$ such that the border of $\mathcal{R}_{q, \varepsilon}^m$ has measure zero for $\xi^{M,m}$, for any $M,m \in \bbN^*$.\\

The motivation for introducing this partition is that it satisfies the hypotheses of Proposition \ref{prop:entropy-ineq}, as will be proved in Lemma \ref{lemma:entropy-R}.
Indeed, until then, we will only use the fact that for any $x,y$ in the same atom of $\mathcal{R}_{q,\varepsilon}$, we have
\begin{align}
\label{eq:distortion-atom}
| \log |f'(x)| - \log |f'(y)|| \leq 1/q.
\end{align}

\section{Entropy of the neutral component}
\label{sc:ENC}

\subsection{General bounds}

By following the first part of the proof of Lemma 9 from \cite{B}, we obtain the following estimate for $H_{\zeta}(\mathcal{P}^n \mid \mathcal{Q}^{E_n^{M,m}})$.

\begin{lemma}[Lemma 9 from \cite{B}]
\label{le:reparamization-partition}
Let $\mathcal{Q}$ be a measurable partition for $\nu$ which satisfies the hypotheses of Proposition \ref{prop:entropy-ineq}.
Assume that for every atom $Q$ of $\mathcal{Q}^{E_n^{M,m}}$, the set $Q \cap \mathtt{F}_{\nu}$ is covered by a family $\Psi_Q$ of affine maps $\theta : [-1,1] \to I$ such that for any $k \in [\![0 ; n-1 ]\!]$, we have $|| (f^k \circ \theta )' ||_{\infty} \leq 1$.
Then, we have
$$
\limsup\limits_{n \to +\infty} \frac{1}{n} H_{\zeta}(\mathcal{P}^n \mid \mathcal{Q}^{E_n^{M,m}})
\leq \limsup\limits_{n \to +\infty} \frac{1}{n}\sum\limits_{Q \in \mathcal{Q}^{E_n^{M,m}}}\zeta(Q) \log \sharp \Psi_Q.
$$
\end{lemma}

Let $C(f):=\frac{\log^{+}||f'||_{\infty}}{r} + \log C_r + 1 + 2\log (1+\log ||f' ||_{\infty})$ where $C_r$ is a constant that depends only on $r$ which we will make precise at the end of subsection \ref{sec:RL}.
Notice that $f \mapsto C(f)$ is continuous in the $\mathcal{C}^1$-topology when $|| f' ||_{\infty} > 1/e$ and that $\lim\limits_{p\rightarrow \infty}\frac{C(f^p)}{p}=\frac{R(f)}{r}$. 
The goal of this section is to prove the following lemma, which we will do in the next subsection.

\begin{lemma}
\label{lemma:repabound}
Let $q \in \bbN^*$.
There exists $\varepsilon_q(f) > 0$ such that, if we let $\mathcal{R}_q$ be the partition defined in section \ref{ssec:def-partition} for $\varepsilon = \varepsilon_q(f)$, then the following property holds. For any $m \in \bbN^*$, for $M$ large enough, and for $n$ large enough,
there exists $\gamma_{n,M,m,q}(f) > 0$ such that, if $F_n$ is the intersection between an atom of $\mathcal{R}_q^{E_n^{M,m}}$ and $\mathtt{F}_{\nu}$, then $F_n$ is covered by a family $\Psi_{F_n}$ of affine maps $\theta : [0,1] \to I$ such that for any $k \in [\![0 ; n-1 ]\!]$, we have $|| (f^k \circ \theta )' ||_{\infty} \leq 1$, and such that 
$$
\begin{aligned}
\frac{1}{n}\log \sharp \Psi_{F_n}\leq &\left( 1-\frac{\sharp E_n^{M,m}}{n} \right)C(f)\\
&+\frac{1}{r-1}\left( \int \frac{\log^+|(f^q)'|}{q}d\xi^{M,m,q}_{n,\nu,F_n}-\int\frac{\log|(f^q)'|}{q}d\xi^{M,m,q}_{n,\nu,F_n} \right)\\
&+ \gamma_{n,M,m,q}(f).
\end{aligned}
$$
where $\xi^{M,m,q}_{n,\nu,F_n}:=\xi^{M,m,q}_{n,\nu_{F_n}}$ and
$\gamma_{n,q,m,M}(f)$ goes to zero uniformly in $f$: for $K > 0$,
    $$
    \limsup\limits_{q \to +\infty} \limsup\limits_{M \to +\infty} \limsup\limits_{n \to +\infty} \sup\limits \{\gamma_{n,M,m,q}(f) : f : I \to I \text{ is } \mathcal{C}^r \text{ and } || f' ||_{\infty} \leq K \} = 0.
    $$
\end{lemma}

\subsection{Semi-local Reparametrization Lemma}
\label{sec:RL}

In this section, we prove Lemma \ref{lemma:repabound} by building a family $\Psi_{F_n}$ of affine maps.
Given $g:I\rightarrow I$, a $\mathcal{C}^r$-map, define
$$
\begin{aligned}
    &k_g(x):=\lfloor\log^+|g'(x)|\rfloor=\lfloor\max\{ \log|g'(x)|,0 \}\rfloor;\\
    &k'_g(x):=\lfloor\log^-|g'(x)|\rfloor=\lfloor-\min\{ \log|g'(x)|,0 \}\rfloor.
\end{aligned}
$$
We first define the notion of bounded reparametrization. We will say that a map $\sigma : [-1 ; 1] \to I$ is a \textit{reparametrization} if it is a $\mathcal{C}^r$ map whose derivative does not vanish.

\begin{definition}[Bounded reparametrization]
A reparametrization $\sigma$ is said to be bounded if
$$
\max\limits_{s \in ]\!] 1 ; [r] ]\!] \cup \{r \}}|| d^s \sigma ||_{\infty} \leq \frac{1}{6} || \sigma' ||_{\infty}.
$$
Then, for $\varepsilon > 0$, it is said to be $\varepsilon$-bounded if it also satisfies
$$
|| \sigma' ||_{\infty} \le \varepsilon.
$$
Moreover, we say it is $(n , \varepsilon)$-bounded for $g : I \to I$ if we have
$$
\forall i \in [ \! [ 0 ; n ] \! ], g^i \circ \sigma \text{ is } \varepsilon \text{-bounded}.
$$
\end{definition}

The key property is that bounded reparametrizations have bounded distortion:

\begin{lemma}[Distortion inequality, Lemma 2.2 from \cite{A}]
\label{le:distortion}
If $\sigma$ is a bounded reparametrization, then we have
$$
\forall t,s \in [-1 ; 1], \frac{| \sigma'(t) |}{|\sigma'(s)|} \leq \frac{3}{2}.
$$
\end{lemma}

The following lemma allows a control of the local dynamics in terms of bounded reparametrizations.

\begin{lemma}[Lemma 12 from \cite{B}]
\label{le:1step}
Let $\varepsilon=\varepsilon(g)>0$ such that for all $x\in I$, for all $s\in [1,r]$,
$$
||d^sg^x_{2\varepsilon}||_{\infty}\leq 3\varepsilon\max\{ 1,|g'(x)| \},
$$ 
where $
g^x_{2\varepsilon}(t):=g(x+2\varepsilon \cdot t).
$
Let $\sigma$ be an $\varepsilon$-bounded reparametrization. Then there exists $C_r>0$ such that for all $k,k'\in\mathbb{N}$, there exists $\Theta$, a family of affine maps from $[-1,1]$ to itself, with
\begin{itemize}
    \item[1.] $\sigma^{-1}\{ x:k_g(x)=k, k'_g(x)=k'\}\subseteq \bigcup_{\theta\in \Theta} \theta([-1,1])$;
    \item[2.] for any $\theta \in \Theta$, $g\circ \sigma \circ \theta$ is bounded;
    \item[3.] for any $\theta \in \Theta$, $|\theta'|\leq  e^{\frac{-k'-1}{r-1}}/4$;
    \item[4.] $\sharp \Theta\leq C_r e^{\frac{k'}{r-1}}$.
\end{itemize}
\end{lemma}

\begin{proof}
The proof of Lemma 12 from \cite{B} works with this choice of $\varepsilon(g)$, for more details, one may also check the proof of Lemma 2.4 from \cite{A}.
\end{proof}

The idea is to iterate this lemma along neutral blocks. This allows us to control the local dynamics at these times, hence to obtain a bound on the entropy production.
To do so, we define a sequence of times $a_1, ..., a_p$ and we successively apply the above lemma for $g = f^{a_1}$, then $g = f^{a_2 - a_1}$, ...
We choose them such that the set of times $a_i$ contains completely the neutral blocks.
Recall that the neutral blocks depend on the orbit.
Therefore, let $F_n$ be the intersection between an atom of $\mathcal{R}_{q,\varepsilon_q}^{E^{M,m}_n} = \mathcal{R}_q^{E^{M,m}_n}$ and $\mathtt{F}_{\nu}$.
Thus, we will write $E_n^{M,m}$ for $E_n^{M,m}(x)$ for any $x \in F_n$.

\bigskip

We would like to use the set $E_n^{M,m+q}(x)$, but it depends on $x$. We define the set $\widetilde{E}_n^{M,m,q}$ to be the set $E_n^{M,m}$ to which we remove the last $q-1$ elements of each connected component. Hence $\widetilde{E}^{M,m,q}_n$ is independent of $x$.
We assume $m \geq q$, so that for any $x \in F_n$, we have
$$
\widetilde{E}_n^{M,m,q} \subset E_n^{M,m} = E_n^{M,m,q}(x).
$$
and these sets differ by at most $q\times \sharp \partial E_n^M /2$ elements within $E_n^{M,m}$.

\begin{lemma}\label{le:q-control}
There exists $c\in [\![0,q[\![$ such that for any $x\in F_n$,
$$
\sum_{l\in (c+q\mathbb{N})\cap \widetilde{E}_n^{M,m,q}} k'_{f^q}(f^l x)\leq \frac{n}{q}+ n\left(\int \frac{\log^+|(f^q)'|}{q}d\xi^{M,m,q}_{n,\nu,F_n}-\int\frac{\log|(f^q)'|}{q}d\xi^{M,m,q}_{n,\nu,F_n}\right).
$$
\end{lemma}
\begin{proof}
First take $c\in [\![0,q[\![$ such that 
$$
\sum_{l\in (c+q\mathbb{N})\cap \widetilde{E}_n^{M,m,q}} k'_{f^q}(f^l x)\leq \frac{1}{q}\sum_{l\in \widetilde{E}^{M,m,q}_n} k'_{f^q}(f^l x).
$$
Then, the purpose of introducing $\widetilde{E}_n^{M,m,q}$ is that for any $l\in \widetilde{E}_n^{M,m,q}$ and $k \in [\![0 ; q [\![$, we have $l+k \in E_n^{M,m}$. Hence, since $F_n$ is contained in an atom of $\mathcal{R}_q^{E^{M,m}_n}$, equation \ref{eq:distortion-atom} gives
\begin{align*}
\sum_{l\in \widetilde{E}^{M,m,q}_n} k'_{f^q}(f^l x)
&\leq \sharp \widetilde{E}^{M,m,q}_n +\int \sum\limits_{l \in \widetilde{E}_n^{M,m,q}} \log^-|(f^q)'(f^l y)| d \nu_{F_n}\\
&\leq n + n\int \log^- |(f^q)'(y)| d\xi^{M,m,q}_{n,\nu,F_n}.
\end{align*}
Then, by definition,
$$
\sum_{l\in (c+q\mathbb{N})\cap \widetilde{E}^{M,m,q}_n} k'_{f^q}(f^l x)\leq \frac{n}{q}+ n\int \frac{\log^+|(f^q)'|}{q}d\xi^{M,m,q}_{n,\nu,F_n}-n\int\frac{\log|(f^q)'|}{q}d\xi^{M,m,q}_{n,\nu,F_n}.$$
\end{proof}

Define 
$$
\begin{aligned}
&\partial_l E^{M,m}_n:=\{ a\in E^{M,m}_n, \text{with $a-1\notin E^{M,m}_n$} \};\\
&\mathcal{A}_n:=\{ 0=a_1<a_2<\cdots <a_p \}:=\partial_l E^{M,m}_n\cup \left((c+q\mathbb{N})\cap[\![0,n[\![\right)\cup \left([\![0,n[\![-E^{M,m}_n\right),
\end{aligned}
$$
then, let $b_i:=a_{i+1}-a_i$ for $i=0,\cdots, p-1$ and $b_p:=n-a_p$.\\

For each $f$, we denote by $\varepsilon_f$ the quantity given by Lemma \ref{le:1step}. Now we choose this $\varepsilon_f$ small so that $\varepsilon_{f^k}\leq \varepsilon_{f^l}\leq \max\{ 1,||f'||_{\infty} \}^{-l}$ for any $q\geq k\geq l\geq 1$. We choose $\varepsilon_q<\frac{\varepsilon_{f^q}}{3}$.\\

Given ${\bf{k}}:=(k_l,k_l')_{l\in \mathcal{A}_n}$; $m_n\in [\![0,n[\![$; $\overline{E}\subseteq [\![0,n[\![$, define $F_n^{({\bf{k}}, \overline{E}, m_n)}$ to be the set of $x\in F_n$ satifying:
\begin{itemize}
    \item[1.] $\overline{E}=E_n(x)\setminus E^{M,m}_n$;
    \item[2.] $k_{a_i}=k_{f^{b_i}}(f^{a_i}x)$, $k'_{a_i}=k'_{f^{b_i}}(f^{a_i}x)$, $i=1,\cdots,p-1$;
    \item[3.]$m_n(x)=m_n$.
\end{itemize}

We recall that we chose $\mathtt{F}_{\nu}$ after the proof of Lemma \ref{le:bn} so that $\frac{m_n^+(x)}{n} \to 1$ uniformly for $x \in \mathtt{F}_{\nu}$. For $t\in (0,1)$, we write
$$
H(t):=t\log\frac{1}{t}+(1-t)\log\frac{1}{1-t}.
$$
Recall that $\binom{n}{k} \leq e^{nH(\frac{k}{n})}$ for any $0 \leq k \leq n$.

\begin{lemma}\label{le:combinatorics}
    For every $m^2 \leq M\leq \frac{n}{2}$ and $M \geq 16$, we have
    $$
    \begin{aligned}
    \frac{1}{n} \log \sharp \{ ({\bf{k}}, \overline{E}, m_n): F_n^{({\bf{k}}, \overline{E}, m_n)} \neq \emptyset \} \leq &\frac{\log n}{n} + H\left ( \frac{2}{\sqrt{M}} \right ) + \frac{1}{q} + \frac{1}{M}
    + \sup\limits_{x \in \mathtt{F}_{\nu}} \frac{m_n^+(x)-n}{n}\\ &+ \left ( 1 - \frac{\sharp E_n^{M,m}}{n} \right )(1+2\log(1+\log ||f' ||_{\infty})).
    \end{aligned}
    $$
\end{lemma}

\begin{proof}
If $\overline{E}$ is such that $F_n^{({\bf{k}}, \overline{E}, m_n)}\neq \emptyset$, then elements in $\overline{E}$ are at a distance at least $M$ from each other, except maybe for the ones we removed to go from $E_n^M$ to $E_n^{M,m}$.
There are at most $m \times \frac{\sharp \partial E_n^{M,m}}{2}$ of these, so that
$$
\sharp\overline{E} \leq \frac{n+M}{M} + \frac{m \sharp \partial E_n^{M,m}}{2}
\leq \frac{2n}{\sqrt{M}}.
$$
Therefore the number of possible $\overline{E}$ making $F_n^{({\bf{k}}, \overline{E}, m_n)}$ nonempty is less than $\binom{n}{ \left \lfloor \frac{2n}{\sqrt{M}} \right \rfloor}$.
Then, $m_n$ belongs to $[\![0 ; n [\![$, therefore the choice of $m_n$ is less than $n$. 
Thus, we have
$$
\begin{aligned}
&\sharp \{ ({\bf{k}}, \overline{E}, m_n): F_n^{({\bf{k}}, \overline{E}, m_n)} \neq \emptyset \}\\
&\leq \sharp\{ {\bf{k}}: \exists x\in F_n, k_{a_i}=k_{f^{b_i}}(f^{a_i}x), k'_{a_i}=k'_{f^{b_i}}(f^{a_i}x) \}\cdot \binom{n}{\left \lfloor \frac{2n}{\sqrt{M}} \right \rfloor}\cdot n,
\end{aligned}
$$
and it suffices to bound
$$
\sharp\{ {\bf{k}}: \exists x\in F_n, k_{a_i}=k_{f^{b_i}}(f^{a_i}x), k'_{a_i}=k'_{f^{b_i}}(f^{a_i}x) \}.
$$

Recall that $F_n$ is contained in an atom of $\mathcal{R}_q^{E_n^{M,m}}$, where $k$ and $k'$ are fixed --- they vary by at most $1$ by equation \ref{eq:distortion-atom}.
Since there are at most $\frac{n}{q} + \frac{n}{M}$ $a_i$'s in $E_n^{M,m}$, it adds a factor $2^{\frac{n}{q} + \frac{n}{M}}$.
Then each $k_{a_i}$ is contained in $[\![0 ; \log ||f' ||_{\infty} ]\!]$, hence the number of $k_{a_i}$'s for $a_i \notin E_n^{M,m}$ is bounded by $(\log || f' ||_{\infty})^{n - \sharp E_n^{M,m}}$.
Therefore, it suffices to bound the following quantity
$$
\sharp \{ (k_{a_i}')_{a_i \notin E_n^{M,m}} :  \exists x\in F_n, \forall a_i \notin E_n^{M,m}, k_{a_i}=k_{f^{b_i}}(f^{a_i}x), k'_{a_i}=k'_{f^{b_i}}(f^{a_i}x) \}.
$$
Notice that when $i$ is such that $a_i \notin E_n^{M,m}$, then $b_i = 1$.
Let $r_n = \sharp \{i :  a_i \notin E_n^{M,m} \} = n - \sharp E_n^{M,m}$ and write
$$
[\![ 0 ; n [\![ \backslash E_n^{M,m} = \bigcup\limits_{j=1}^l [\![c_j ; d_j [\![.
$$
where $[\![c_j ; d_j [\![$ are the connected components of $[\![ 0 ; n [\![ \backslash E_n^{M,m}$.
Let $x \in F_n$ and let $k'_{a_i}=k'_{f^{b_i}}(f^{a_i}x)$ for all $a_i \notin E_n^{M,m}$.
By definition of $k'_{g}$,
\begin{align*}
\sum\limits_{i \text{ s.t. } a_i \notin E_n^{M,m}} k'_{a_i}&=\sum\limits_{i \text{ s.t. } a_i \notin E_n^{M,m}} \lfloor\log^{-} |f'(f^{a_i}x)|\rfloor\\
&\leq \sum\limits_{i \text{ s.t. } a_i \notin E_n^{M,m}} \left(\log||f'||_{\infty} -\log|f'(f^{a_i}x)| \right)\\
&=r_n \log || f' ||_{\infty} - \sum\limits_{j=1}^l \log |(f^{d_j - c_j})'(f^{c_j} x)|.
\end{align*}
Then, notice that $d_j \in E(x)$ except maybe for $j = l$, the last one, for which we may have $d_l = n$ and $n \notin E(x)$.
Therefore, for $j < l$, we have
$$
\log |(f^{d_j - c_j})'(f^{c_j} x)| \geq \psi_{d_j-c_j}(f^{c_j} x) \geq \delta(d_j - c_j) \geq 0.
$$
Then, if $d_l \in E(x)$, the same argument gives
$$
\sum_{i \text{ s.t. } a_i \notin E_n^{M,m}} k'_{a_i}
\leq r_n \log || f' ||_{\infty}.
$$
Otherwise, we have $d_l = n$ and
\begin{align*}
\sum_{i \text{ s.t. } a_i \notin E_n^{M,m}} k'_{a_i}
&\leq r_n \log || f' ||_{\infty} - \log |(f^{n-c_l})'(f^{c_l} x)|\\
&\leq r_n \log || f' ||_{\infty} - \log |(f^{m_n^+(x)-c_l})'(f^{c_l} x)| + \log |(f^{m_n^+(x) - n})'(f^n x)|\\
&\leq r_n \log || f' ||_{\infty} + (m_n^+(x) - n) \log || f' ||_{\infty}.
\end{align*}
Let $m_n^+ = \sup\limits_{x \in \mathtt{F}_{\nu}} m_n^+(x)$. The previous bounds give
\begin{align*}
\frac{1}{n} \log \sharp \{(k_{a_i}')_{a_i \notin E_n^{M,m}}\} &\leq \frac{1}{n} \log {{(r_n+m_n^+-m_n)\log ||f'||_{\infty} + r_n}\choose{r_n}}\\
\text{from } {{a}\choose{b}}\leq \left ( \frac{ae}{b} \right )^b \hspace{1em}&\leq \frac{r_n}{n} \log \frac{((r_n+m_n^+-m_n)\log ||f'||_{\infty} + r_n)e}{r_n}\\
&\leq \frac{r_n}{n}\left ( 1 + \log(1 + \log|| f' ||_{\infty}) + \log \frac{r_n + m_n^+ - m_n}{r_n}\right )\\
&\leq \frac{r_n}{n}\left ( 1 + \log(1 + \log|| f' ||_{\infty}) + \frac{m_n^+ - m_n}{r_n}\right )\\
&=\left ( 1 - \frac{\sharp E_n^{M,m}}{n} \right )\left ( 1 + \log(1 + \log|| f' ||_{\infty}) \right ) + \frac{m_n^+ - m_n}{n}
\end{align*}

\end{proof}

Fix ${\bf{k}}, m_n$ and $\overline{E}$.
The following reparametrization lemma is obtained by iterating Lemma \ref{le:1step} and is the main statement to prove Lemma \ref{lemma:repabound}, hence to bound the entropy production along neutral blocks.

\begin{lemma}\label{le:repara}
With the above notations, there are families $(\Theta_i)_{1 \leq i \leq p}$ of affine maps from $[-1,1]$ to itself such that
\begin{itemize}
    \item[1.] For each $i \leq p$, we have $\sigma^{-1}\left(F_n^{({\bf{k}}, \overline{E}, m_n)} \right)\subseteq \bigcup_{\theta\in \Theta_i} \theta([-1,1])$;
    \item[2.] for all $\theta\in \Theta_i$, for all $j\leq i$, $f^{a_j}\circ \sigma \circ \theta$ is strongly $\varepsilon_{f^{b_j}}$-bounded;
    \item[3.] for all $\theta\in \Theta_i$, for all $j<i$, there exists $\theta^i_j\in \Theta_j$, such that
    $$
    \frac{|(\theta_i)'|}{|(\theta^i_j)'|}\leq \prod_{j\leq l<i}\left ( e^{\frac{-k'_{a_l}-1}{r-1}} / 4 \right );
    $$
    \item[4.]
    for all $i=1,\cdots, p$,
    $$
    \sharp\Theta_i\leq 
    \max\{ 1,||f'||_{\infty} \}^{\sharp(\overline{E}\cap [\![1,a_i]\!])}\prod_{j<i}C_r e^{\frac{k'_{a_j}}{r-1}}.
    $$
\end{itemize}
\end{lemma}
\begin{proof}
We use the same argument as for the proof of Lemma 15 in \cite{B}. One difference is that we obtain a factor $\log(16/9)$ in our calculations while Burguet obtained a $\log 2$. This slightly affects our choice of $\delta$ in section \ref{ssec:main-prop}.
\end{proof}

\begin{remark}
The above statement is not exactly the same as Burguet's Lemma 15 in \cite{B}.
Indeed, in the cardinality bound from item $4$, he has an additional multiplicative constant $C$ on the right-hand side.
The reason is mostly technical and comes from the fact that we do not use the same definition for the set of hyperbolic times $E_{\delta}(x)$.
\end{remark}

The following lemma gives the main term of Lemma \ref{lemma:repabound}.

\begin{lemma}\label{le:firstterm}
$$\sum_{i, \ m_n>a_i\notin \widetilde{E}_n^{M,m,q}}\frac{k'_{a_i}}{r-1}\leq \left(n-\sharp \widetilde{E}_n^{M,m,q}\right) \frac{\log^+ \|f' \|_\infty}{r}.
$$ 
\end{lemma}
\begin{proof}
Fix $x\in F_n$. 
Write
$$
\{i: m_n>a_i\notin \widetilde{E}_n^{M,m,q} \} = \bigcup\limits_{j = 1}^s [\![ c_j ; d_j [\![
$$
where $[\![c_j ; d_j [\![$ are the consecutive $i$'s such that $a_i \notin \widetilde{E}_n^{M,m,q}$. Notice that $a_{d_j}\in E(x)$ for $j=1,\ldots, s$.
For $a,b >0$, we have $\log^+ (ab) \leq \log^+ a + \log^+ b$, thus
$$
\begin{aligned}
\sum_{i, \ m_n>a_i\notin \widetilde{E}_n^{M,m,q}} \psi_{b_i}(f^{a_i} x)&\leq\sum_{i, \ m_n>a_i\notin \widetilde{E}_n^{M,m,q}}\left( \log|(f^{b_i})'(f^{a_i}x)|-\frac{1}{r}\log^+|(f^{b_i})'(f^{a_i}x)| \right)\\
&=\sum_{i, \ m_n>a_i\notin \widetilde{E}_n^{M,m,q}}\left( (1-\frac{1}{r})\log^+|(f^{b_i})'(f^{a_i}x)|-\log^-|(f^{b_i})'(f^{a_i}x)| \right)\\
&\leq \frac{r-1}{r}(n-\sharp \widetilde{E}_n^{M,m,q})\log^+||f'||_{\infty}-\sum_{i, \ m_n>a_i\notin \widetilde{E}_n^{M,m,q}} k_{a_i}'.
\end{aligned}
$$
Therefore
$$
\begin{aligned}
\sum_{i, \ m_n>a_i\notin \widetilde{E}_n^{M,m,q}} \frac{k_{a_i}'}{r-1}&\leq\frac{1}{r}(n-\sharp \widetilde{E}_n^{M,m,q})\log^+||f'||_{\infty}-\frac{1}{r-1}\sum_{i, \ m_n>a_i\notin \widetilde{E}_n^{M,m,q}} \psi_{b_i}(f^{a_i} x)\\
&\leq \frac{1}{r}(n-\sharp \widetilde{E}_n^{M,m,q})\log^+||f'||_{\infty}.
\end{aligned}
$$
The last inequality holds, since $a_{d_j}\in E(x)$ implies
$$
\sum_{i, \ m_n>a_i\notin \widetilde{E}_n^{M,m,q}} \psi_{b_i}(f^{a_i} x)=\sum_{j=1}^s \psi_{a_{d_j}-a_{c_j}}(f^{a_{c_j}}x)\geq 0.
$$
\end{proof}

\bigskip

We conclude this section by the proof of Lemma \ref{lemma:repabound}

\begin{proof}[Proof of Lemma \ref{lemma:repabound}]
We let $\Psi_{F_n}$ be a family $\Theta_p$ of reparametrizations as given in Lemma \ref{le:repara}.
Then we decompose the bound on $\sharp \Theta_p$ given by item $4)$ of Lemma \ref{le:repara} into four terms:
\begin{itemize}
\item[1.] by Lemma \ref{le:firstterm}, 
$$\sum_{i,m_n>a_i\notin \widetilde{E}_n^{M,m,q}} \frac{k'_{a_i}}{r-1}\leq (n-\sharp \widetilde{E}_n^{M,m,q})\frac{\log^{+}||f'||_{\infty}}{r};$$
\item[2.] Let $x \in F_n^{({\bf{k}}, \overline{E}, m_n)}$, so that $m_n=m_n(x)=\max E(x)\cap [\![0,n[\![$.
Therefore, we have
\begin{align*}
\hspace{-3em}\sum_{i, m_n\leq a_i} \frac{k'_{a_i}}{r-1} &\leq\sum_{k=m_n}^{n-1} \frac{1}{r-1} \log^- |f'(f^k x)|\\
&\leq \frac{1}{r-1}\sum_{k=m_n}^{n-1} \left(\log||f'||_{\infty} -\log|f'(f^k x)| \right)\\
&=\frac{1}{r-1}\left((n-m_n)\log ||f'||_{\infty}-\left( \log|(f^{m_n^{+}(x)-m_n})'(f^{m_n}x)|- \log|(f^{m_n^+(x)-n})'(f^n x)|\right)\right)\\
&\leq\frac{1}{r-1}\left( (n-m_n)\log ||f'||_{\infty}+(m^+_n(x)-n)\log ||f'||_{\infty}\right)\\
&=\frac{1}{r-1}(m^+_n(x)-m_n)\log ||f'||_{\infty}.
\end{align*}
\item[3.] by Lemma \ref{le:q-control},
$$\sum_{i,a_i\in \widetilde{E}_n^{M,m,q}\cap (c+q\mathbb{N})} \frac{k'_{a_i}}{r-1}\leq \frac{n}{r-1}\left( \frac{1}{q}+ \int \frac{\log^+|(f^q)'|}{q}d\xi^{M,m,q}_{n,\nu,F_n}-\int\frac{\log|(f^q)'|}{q}d\xi^{M,m,q}_{n,\nu,F_n}
 \right);$$
\item[4.] if $a_i\in \widetilde{E}^{M,m,q}_n\setminus (c+q\mathbb{N})$, then $a_i\in \partial_l E^{M,m}_n$ and $[\![a_i ; a_i+b_i [\![ \subset E_n^{M,m}$, therefore Lemma \ref{lemma:prop-hb} gives
$$
\begin{aligned}
\sum_{i,a_i\in \widetilde{E}^{M,m,q}_n\setminus (c+q\mathbb{N})} \frac{k'_{a_i}}{r-1}&\leq \sum_{i,a_i\in \partial_l E^{M,m}_n} \frac{k'_{a_i}}{r-1}\\
&\leq \frac{1}{r-1}\sum_{i,a_i\in \partial_l E^{M,m}_n} \log^{-} |(f^{b_i})'(f^{a_i} x)|\\
&\leq \frac{1}{r-1}\frac{\sharp\partial E^{M,m}_n}{2}\cdot M\cdot \log||f'||_{\infty}.
\end{aligned}
$$
\end{itemize}
We define
$$
\begin{aligned}
&\gamma_{n,M,m,q}(f):=
\left ( \left(\frac{M}{r-1} + \frac{q}{r} + \log C_r \right)\frac{\sharp\partial {E^{M,m}_n}}{2n} + \frac{2}{\sqrt{M}} + \sup\limits_{x \in \mathtt{F}_{\nu}} \frac{m_n^+(x) - m_n(x)}{n(r-1)} \right )\cdot \log||f'||_{\infty}\\
&\hspace*{6em}+ H(\frac{2}{\sqrt{M}})+\frac{\log C_r}{q} + \frac{\log n}{n} + \frac{r}{q(r-1)}+\sup_{x\in \mathtt{F}_{\nu}}\frac{m_n^+(x)-n}{n}+\frac{1}{M};\\
&C(f):=\frac{\log^{+}||f'||_{\infty}}{r} + \log C_r + 1 + 2\log (1+\log ||f' ||_{\infty}).
\end{aligned}
$$
Then, by noting that $\sharp \mathcal{A}_n \leq n - \sharp E_n^{M,m} + \frac{n}{q} + \frac{n}{M}$
and by Lemma \ref{le:combinatorics}, we obtain
$$
\begin{aligned}
\frac{1}{n}\log \sharp \Psi_n\leq &\left( 1-\frac{\sharp E^{M,m}_n}{n} \right)C(f)\\
&+\frac{1}{r-1}\left( \int \frac{\log^+|(f^q)'|}{q}d\xi^{M,m,q}_{n,\nu}-\int\frac{\log|(f^q)'|}{q}d\xi^{M,m,q}_{n,\nu} \right)\\
&+\gamma_{n,M,m,q}(f).
\end{aligned}
$$
Then, if $\theta$ is in $\Psi_n$, then for $i = 1, ..., p$, the choice of $\varepsilon_{f^{b_i}}$ gives
$$
|| (f^{a_i} \circ \sigma \circ \theta)' ||_{\infty} \leq \varepsilon_{f^{b_i}} \leq \max \{ 1, || f' ||_{\infty} \}^{- b_i}.
$$
Therefore, for $j=0,...,n$, we have
$
|| (f^j \circ \sigma \circ \theta)' ||_{\infty} \leq 1.
$
\end{proof}

\section{Entropy and exponent of the hyperbolic component}
\label{sec:main-prop}

\subsection{Main Proposition}
\label{ssec:main-prop}

Let $f : I \to I$ be a $\mathcal{C}^r$ map of positive topological entropy.
We now consider a sequence of maps $(f_k)$ converging $\mathcal{C}^r$-weakly to $f$ and a sequence of measures $(\nu_k)$ such that each $\nu_k$ is ergodic for $f_k$ and $(\nu_k)$ converges to some measure $\mu$ in the weak-$\ast$ topology.
The $f_k$'s will be the $f$'s of the previous sections.
Instead of proving Theorem \ref{th:main1}, we prove the following more technical statement:

\begin{theorem}\label{th:main2}
Let $(f_k)_{k\in \mathbb{N}}$ and $(\nu_k)_{k\in \mathbb{N}}$ be as above.
Let $\mathcal{K}$ be a sequence of integers such that $(h_{f_k}(\nu_k))_k$ converges to its limsup along $\mathcal{K}$.
For all $\alpha > \frac{R(f)}{r}$, there exist two $f$-invariant probability measures $\mu_1$, $\mu_0$, and $\beta\in [0,1]$, with $\mu = \beta \mu_1 + (1-\beta) \mu_0$ and
\begin{itemize}
\item[\namedlabel{itm:main21}{1)}] $\beta>0$ if and only if $\liminf\limits_{\mathcal{K} \ni k \to +\infty} \lambda_{f_k}(\nu_k) > \alpha$;
\item[\namedlabel{itm:main22}{2)}] $\limsup\limits_{k \to +\infty} h_{f_k}(\nu_k) \leq \beta\cdot h_f(\mu_1)+(1-\beta)\cdot\alpha$;

\item[3)] assume that $\beta>0$ and that $p \in \bbN^*$ is such that
$$
\liminf\limits_{\mathcal{K} \ni k \to +\infty} \lambda_{f_k}(\nu_k)
> \alpha - \frac{R(f)}{r} + \frac{1}{rp} \log^+ || (f^p)' ||_{\infty}.
$$
Then, by letting $\delta=\min\{ \alpha-\frac{R(f)}{r}, \frac{r-1}{r p} \log (16/9) \}$, we have
\begin{itemize}
    \item[\namedlabel{itm:main21'}{1')}] for $\mu_1$-a.e. $x \in I$, we have $\lambda_f(x)\geq \delta$;
    \item[\namedlabel{itm:main23'}{3')}] $\liminf\limits_{\mathcal{K} \ni k \rightarrow \infty} \lambda_{f_k}(\nu_k)\geq \beta\lambda_f(\mu_1) + (1-\beta) \delta$.
\end{itemize}
\end{itemize}
\end{theorem}
Notice that if $0 < \beta < 1$, then from upper semi-continuity of $\mu \mapsto \lambda_f(\mu)$, item \textit{3')} gives $\lambda_f(\mu_0) \geq \delta$.\\

One can see that this implies Theorem \ref{th:main1}.
In order to prove Theorem \ref{th:main2}, we prove, in section \ref{section:proof-th}, that some iterate of $f$ satisfies the hypotheses of the following main Proposition \ref{prop:main-prop}, whose proof will be completed in section \ref{ssec:proof}.
Finally, in section \ref{section:proof-th} as well, we show that this suffices to get Theorem \ref{th:main2}.
Meanwhile, section \ref{sec:entropy-HB} provides the last required ingredients to prove the main proposition.

\begin{proposition}
\label{prop:main-prop}
Let $f_k \to f$ $\mathcal{C}^r$-weakly and let $(\nu_k)$ be a sequence of $f_k$-ergodic measures that converges to $\mu$ in the weak-$*$ topology.
Let $\mathcal{K}$ be a sequence of integers such that $(h_{f_k}(\nu_k))_k$ converges to its limsup along $\mathcal{K}$.
Assume that $h_{\rm{top}}(f) > 0$. If $\alpha > \frac{R(f)}{r}$ and if 
\begin{itemize}
\item[\namedlabel{itm:main-prop1}{1)}] $\liminf\limits_{\mathcal{K} \ni k \to +\infty} \lambda_{f_k}(\nu_k) > \alpha - \frac{R(f)}{r} + \frac{1}{r} \log^+ ||f'||_{\infty}$;
\item[\namedlabel{itm:main-prop2}{2)}] $C(f) < \alpha$,
\end{itemize}
where $C(f)$ is the quantity of Lemma \ref{lemma:repabound}.
Then there exist $\beta, \mu_0, \mu_1$ as in Theorems \ref{th:main1} and \ref{th:main2} with $\beta > 0$ and $p=1$ in the definition of $\delta$.
\end{proposition}

In the rest of this section, we detail the notations of the previous sections, now that we have the additional parameter $k$. We will also assume that $k$ always belongs to $\mathcal{K}$.
For $x \in I$, define $\psi^k: I\rightarrow \mathbb{R}\cup \{ -\infty \}$ by
$$
\psi^k := \log |f_k'| - \frac{1}{r} \log^+|f_k'|.
$$
Hence, from Lemma \ref{lemma:delta}, whenever $\delta_k$ is such that $0 < \delta_k < \int \psi^k d \nu_k$, then the set $E_k(x) = E_{k, \delta_k}(x)$ has positive lower density for $\nu_k$-almost every $x \in I$. This density is uniformly lower bounded, as noted in Remark \ref{rmk:density-E}.\\

Note that hypothesis $1)$ from Proposition \ref{prop:main-prop} implies
$$
\liminf\limits_{k \to +\infty} \int \psi^k d \nu_k > \alpha - \frac{R(f)}{r}.
$$
Therefore, we fix $\delta = \alpha - \frac{R(f)}{r}$, so that for $k$ large enough, the set $E_k(x)$ defined for $\delta_k = \delta$ has positive lower density for $\nu_k$-almost every $x \in I$.
In other words, the hypothesis of Proposition \ref{prop:main-prop} allows us to choose $\delta$ independent from $k$.
However, we mentioned in the proof of lemma \ref{le:repara} that $\delta$ must satisfy an additional requirement, so we actually define
$$
\delta = \min \left \{ \alpha - \frac{R(f)}{r}, \frac{r-1}{r}\log (16/9) \right \}.
$$

We will write $E_k(x)$, $E_{k,n}^{M,m,q}(x)$, $\xi_{k,n}^{M,m,q}$, $\beta_{k,n}^{M,m,q}$ and $\eta_{k,n}^{M,m,q}$.
From Cantor's diagonal argument, we may take some subsequences and assume the convergence to the following limits:
\begin{align*}
&\beta_{k,n}^{M,m,q} \underset{n \to +\infty}{\longrightarrow} \beta_k^{M,m,q} \underset{k \to +\infty}{\longrightarrow} \beta^{M,m,q} \underset{M \to +\infty}{\longrightarrow} \beta;\\
&\xi_{k,n}^{M,m,q} \underset{n \to +\infty}{\longrightarrow} \xi_k^{M,m,q} \underset{k \to +\infty}{\longrightarrow} \xi^{M,m,q} \underset{M \to +\infty}{\longrightarrow} \beta \mu_1;\\
&\eta_{k,n}^{M,m,q} \underset{n \to +\infty}{\longrightarrow} \eta_k^{M,m,q} \underset{k \to +\infty}{\longrightarrow} \eta^{M,m,q} \underset{M \to +\infty}{\longrightarrow} (1-\beta)\mu_0.
\end{align*}
The fact that the last column of limits does not depend on $m,q$ comes from item \ref{itm:EnM2} of Lemma \ref{lem:EnM}.
Therefore, we have $\mu = \beta \mu_1 + (1-\beta) \mu_0$. We also write $\xi^M = \xi^{M,1,1}$.

\begin{lemma}
\label{lemma:f-inv}
The measures $\mu_0$ and $\mu_1$ are $f$-invariant probabilities.
\end{lemma}

\begin{proof}
The measures $\mu_1$ and $\mu_0$ are both probabilities by definition.\\

We first prove that $\mu_1$ is $f$-invariant. Let $\varphi : I \to \bbR$ be a continuous function. Let $\varepsilon > 0$. Let $\delta > 0$ be obtained from uniform continuity of $\varphi$ for $\varepsilon$. For $k$ large enough, we have $|| f - f_k ||_{\infty} < \delta$, therefore
\begin{align*}
\hspace{3em}&\hspace{-3em}|\varphi(f_* \mu_1) - \varphi(\mu_1)|\\
&= \lim\limits_{M \to +\infty} \lim\limits_{k \to +\infty} \lim\limits_{n \to +\infty} \frac{1}{n} \left |\int \sum\limits_{i \in E_{k,n}^M(x)} \varphi(f \circ f_k^i x) d \nu_k(x) - \int \sum\limits_{i \in E_{k,n}^M(x)} \varphi(f_k^i x) d \nu_k(x) \right |\\
&\leq \lim\limits_{M \to +\infty} \lim\limits_{k \to +\infty} \lim\limits_{n \to +\infty} ||\varphi ||_{\infty} \int \frac{1}{n}\sharp\partial E_{k,n}^M(x) d \nu_k(x) + \varepsilon\\
\underset{\text{Lemma } \ref{lem:EnM}.\ref{itm:EnM2}}&{=} \varepsilon.
\end{align*}
Since $\varepsilon$ can be chosen arbitrarily close to 0, we have that $\mu_1$ is $f$-invariant. Then the formula $\mu = \beta \mu_1 + (1-\beta) \mu_0$ and the fact that $\mu$ is $f$-invariant concludes that $\mu_0$ is also $f$-invariant.
\end{proof}

Then, from the fact that $E_n^M(x)$ is non-decreasing in $M$, one can obtain the following statement.

\begin{lemma}[Proposition 5.2 from \cite{A}]
\label{lemma:monotone-xi}
For any $m \in \bbN^*$ and Borel set $B \subset I$, we have $\xi^{M,m}(B) \underset{M \to +\infty}{\nearrow} \beta \mu_1(B)$.
\end{lemma}

\subsection{Entropy of the hyperbolic component}
\label{sec:entropy-HB}

From Proposition \ref{prop:entropy-final}, the desired upper bound for the entropy of $\nu$ comes from the sum of the entropy of the hyperbolic component and of the entropy conditioned to this hyperbolic component.
We estimated this second term in section \ref{sc:ENC}, which led to Lemma \ref{lemma:repabound}, and we estimate the first term in this section.\\

In section \ref{ssec:def-partition}, we defined a partition $\mathcal{R}_q$ satisfying the hypotheses of Proposition \ref{prop:entropy-final} and such that $f$ has bounded distortion on each atom and these atoms are small enough.
We first do the same construction for the $f_k$'s and for $f$.\\

For $q\in \mathbb{N}$ and $a\in ]-\frac{1}{q},0[$, we define $I_q$ and $\mathcal{Q}_{q,k}$ (resp. $\mathcal{Q}_{q}$) as in section \ref{ssec:def-partition} for $f_k$ (resp. for $f$).
Then, for each $k$, Lemma \ref{lemma:repabound} gives a scale $\varepsilon_q(f^k)$ that depends on $k$.
However, from the definition of $\varepsilon_q(f_k)$ from Lemma \ref{le:1step} and from the fact that $(f_k)$ is bounded for the $\mathcal{C}^r$ -norm, we may choose $\varepsilon_q$ independent of $k$.
Therefore, let
$$
\mathcal{R}_{q,k} = \mathcal{Q}_{q,k} \vee \mathcal{I}(\varepsilon_q) \;\;\; \text{and} \;\;\; \mathcal{R}_q = \mathcal{Q}_q \vee \mathcal{I}(\varepsilon_q).
$$
where $\mathcal{I}(\varepsilon_q)$ is, as in section \ref{ssec:def-partition}, a partition whose atoms are intervals of diameter less than $\varepsilon_q$.
Choose $a$ and $\mathcal{I}(\varepsilon_q)$ such that the border of $\mathcal{R}_{q,k}^m$ has zero $\mu_1, \xi^{M,m}, \xi_k^{M,m}$-measure, for any $k,M,m \in \bbN^*$.\\

The goal of this section is to prove the following result:

\begin{proposition}
\label{prop:entropy-cv}
For any $q \in \bbN^*$, we have
$$
\lim\limits_{m \to +\infty} \lim\limits_{M \to +\infty} \lim\limits_{k \to +\infty} \frac{1}{m} H_{\overline{\xi_k^{M,m}}}(\mathcal{R}_{q,k}^m) = h_f(\mu_1, \mathcal{R}_{q}) \leq h_f(\mu_1).
$$
\end{proposition}

We start by giving some properties of the partitions defined above.

\begin{lemma}[Proposition 6.6 from \cite{A}]
\label{lemma:entropy-R}
For any $q,k$, the collections $\mathcal{R}_{q,k}$ and $\mathcal{R}_q$ satisfy:
\begin{itemize}
\item[i)] the collection $\mathcal{R}_q$ is a partition for $\mu_1$;
\item[ii)] $H_{\mu_1}(\mathcal{R}_q) < +\infty$;
\item[iii)] $\forall E \in \mathcal{E}_{k,n}^{M,m}, \forall i \in E, f_k^{-i} \mathcal{R}_{q,k}$ is a partition for $(\nu_k)_E$;
\item[iv)] $\forall M,m,k, \sharp \{ R \in \mathcal{R}_{q,k}^m : \exists n, \exists E \in \mathcal{E}_n^{M,m}, \nu_k^E(R) > 0 \} < +\infty$;
\item[v)] $\forall m,M \in \bbN^*, \overline{\xi_k^{M,m}}(\partial \mathcal{R}_{q,k}^m)=0$;
\item[\namedlabel{itm:entropy-R6}{vi)}] $\limsup\limits_{k \to +\infty}\limsup\limits_{n \to +\infty} \frac{1}{n} \sum\limits_{E \in \mathcal{E}_{k,n}^{M,m}} \nu_k(E) \log \left ( \sharp(\mathcal{R}_{q,k})_{\nu_k^E} \right ) \sharp\partial E \underset{M \to +\infty}{\longrightarrow} 0$.
\end{itemize}
\end{lemma}

\begin{remark}
Item $iv)$ does not appear directly in Proposition 6.6 from \cite{A}, but it can be shown with the same argument, as explained in the part 2) of the proof of Proposition 7.2 from \cite{A}. Nonetheless, a more complicated version of this argument is given in Lemma \ref{lemma:CV-k-entropy}, where we show that this finiteness is also uniform on $k$.
\end{remark}

The following lemma is a general technical statement related to the weak-$*$ convergence of measures.

\begin{lemma}
\label{lemma:cv-multi-unif-mes}
Let $X$ and $Y$ be metric spaces with $X$ compact.
Let $(\rho_k)$ be a sequence of probability measures on $X$ that converges to some $\rho$ in the weak-$*$ topology. Consider finitely many sequences of continuous maps $(h_{k,i} : X \to Y)_k$ for $i \in [\![ 1 ; m ]\!]$ such that each converges uniformly to some $h_i$.
Then, for any borelians $A_1, ..., A_m \subset Y$ such that $\rho( h_i^{-1}( \partial A_i)) = 0$ for every $i$, we have
$$
\rho_k(\bigcap_i h_{k,i}^{-1}(A_i)) \underset{k \to +\infty}{\longrightarrow} \rho(\bigcap_i h_i^{-1} (A_i)).
$$
\end{lemma}

\begin{proof}
For $k \in \bbN$, consider the map
$$
h_k = (h_{k,1}, ..., h_{k,m}) : X \to Y^m.
$$
Hence $(h_k)$ converges uniformly to the map $h$ defined as
$$
h = (h_1, ..., h_m) : X \to Y^m.
$$
With these notations, for any borelians $A_1, ..., A_m \subset Y$, we have
$$
\bigcap_i h_{k,i}^{-1}(A_i) = h_k^{-1}(A_1 \times ... \times A_m).
$$
Then, by continuity of the map $(h, \rho) \mapsto h_* \rho$, we have
$$
(h_k)_* \rho_k \overset{*}{\underset{k \to +\infty}{\longrightarrow}} h_* \rho.
$$
Therefore, we are only left to show that for any borelians $A_1, ..., A_m \subset Y$, if $\rho\left ( h_i^{-1}(\partial A_i) \right ) = 0$ for every $i$, then $\rho\left ( h^{-1}\left(\partial \left ( A_1 \times ... \times A_m \right ) \right) \right ) = 0$.
This can be proved by induction on $m$ by using the following fact:
$$
\partial (A \times B) \subset (\overline{A} \times \partial B) \cup ((\partial A) \times \overline{B}).
$$
\end{proof}

We now give the consequences of this lemma regarding the convergence of the entropy of the hyperbolic component.

\begin{lemma}
\label{lemma:CV-k-entropy}
For any $M,m,q \in \bbN^*$, we have
$$
H_{\xi_k^{M,m}}(\mathcal{R}_{q,k}^m) \underset{k \to +\infty}{\longrightarrow} H_{\xi^{M,m}}(\mathcal{R}_q^m).
$$
\end{lemma}

\begin{proof}
Let us write $\psi(x) = -x \log(x)$ for $x \in [0,1]$, so that
\begin{align*}
H_{\xi_k^{M,m}}(\mathcal{R}_{q,k}^m) = \sum\limits_{R \in \mathcal{R}_{q,k}^m} \psi(\xi_k^{M,m}(R)). 
\end{align*}
We first show that only finitely many atoms will appear in the above sum. More precisely, we show that
\begin{align*}
&\sharp \{ ((J_s)_{s \in [\![ 0;m-1]\!]}, (I_s)_{s \in [\![ 0;m-1]\!]}) \in I_q^m\times \mathcal{I}(\varepsilon_q)^m :\\
&\hspace{1em}\exists k \in \bbN, \text{ the associated atom of } \mathcal{R}_{q,k}^m \text{ has positive } \xi_{k}^{M,m} \text{ measure} \} < +\infty.
\end{align*}
Then, we have that $\partial \mathcal{R}_{q,k}^m$ is of zero $\xi_{k}^{M,m}$-measure and $\xi_{k,n}^{M,m} \underset{n \to +\infty}{\longrightarrow} \xi_k^{M,m}$ in the weak-$*$ topology. Thus, it suffices to show that
\begin{align*}
&\sharp \{ ((J_s)_{s \in [\![ 0;m-1]\!]}, (I_s)_{s \in [\![ 0;m-1]\!]}) \in I_q^m\times \mathcal{I}(\varepsilon_q)^m :\\
&\hspace{1em}\exists k,n \in \bbN, \text{ the associated atom of } \mathcal{R}_{q,k}^m \text{ has positive } \xi_{k,n}^{M,m} \text{ measure} \} < +\infty.
\end{align*}
We show that this is a consequence of Lemma \ref{lemma:prop-hb} and of the choice of $E_{k,n}^{M,m}$ instead of $E_{k,n}^{M}$.
Let $R$ be an atom of $\mathcal{R}_{q,k}^m$ and write
$R = \bigcap\limits_{s = 0}^{m-1} f_k^{-s} (I_s \cap (\log |f_k'|)^{-1} (J_s))$ with $J_s \in I_q$ and $I_s \in \mathcal{I}(\varepsilon_q)$.
Assume that there exist $k$ and $n$ with $\xi_{k,n}^{M,m}(R) > 0$. 
By definition of $\xi_{k,n}^{M,m}$, this implies that there exist $x \in I$ and $l \in E_{k,n}^{M,m}(x)$ such that $f_k^l(x) \in R$.
Let $s \in [\![ 0 ; m-1 ]\!]$ and write $J_s = \left [ \frac{i}{q} +a ; \frac{i+1}{q} + a \right [$, we have $\log |f_k'(f_k^{l+s} x)| \leq \frac{i+1}{q} + a \leq \frac{i+1}{q}$.
Then $l$ is in $E_{k,n}^{M,m}(x)$, which implies that $l+s$ is in $E_{k,n}^M(x)$.
Hence, from Lemma \ref{lemma:prop-hb}, we have
$$
\log |f_k'(f_k^{l+s} x)| \geq -M \log || f_k' ||_{\infty}.
$$
In particular, by choosing $k$ large enough depending only on $M$, we have
$$
\log |f_k'(f_k^{l+s} x)| \geq -2M \log || f' ||_{\infty}.
$$
This implies that $i \geq -1 - 2qM \log || f' ||_{\infty}$.
Therefore, there are only finitely many $J_0, ..., J_{m-1} \in I_q$ that give an atom of $\mathcal{R}_{q,k}^m$ of positive $\xi_{k,n}^{M,m}$-measure, for $k$ large enough depending only on $M$.\\

To conclude, we show that each term of the first sum converges. 
We let
\begin{align*}
&A_k^M = \{ x \in I : \log |f_k'(x)| \geq -M\log || f_k' ||_{\infty} \},\\
&A^M = \{x \in I : \log |f'(x)| \geq - M \log || f' ||_{\infty} \}.
\end{align*}
We show that for $k$ large enough, we have $f_k^{-s} A_k^M \subset f^{-s}A^{2M}$ for any $s \in [\![ 0 ; m-1 ]\!]$.
Since $f_k \to f$ $\mathcal{C}^r$-weakly, we have that $f_k' \to f'$ uniformly on $I$.
Let $\eta > 0$.
Let $\varepsilon > 0$ be such that, for any $s > 0$ and $t > || f' ||_{\infty}^{-2M}$ satisfying $|s-t| < \varepsilon$, we have $|\log s - \log t| < \eta$.
Let $k$ be large enough so that $|| f_k' \circ f_k^s -  f' \circ f^s ||_{\infty} < \varepsilon$ for any $s \in [\![0 ; m-1 ]\!]$ and such that $-M \log || f_k'||_{\infty} \geq -2M \log || f' ||_{\infty} + \eta$.
Therefore for such $k$, we have $f_k^{-s} A_k^M \subset f^{-s} A^{2M}$. By the same argument, we also get $f^{-s} A^{2M} \subset f_k^{-s} A_k^{3M}$ for $k$ large enough.\\

Hence, if we let $X = \bigcap\limits_{s = 0}^{m-1} f^{-s} A^{2M}$ and $h_{k,s}(x) = (f_k^s x, \log |f_k'(f_k^sx)|)$ for $s \in [\![0;m-1 ]\!]$ and $x \in I$, then we have
\begin{itemize}
\item[i)] On $X$, $h_{k,s}$ is continuous and converges uniformly to $h_s(x) = (f^s x, \log |f'(f^s x)|)$
\item[ii)] Any $(I_s),(J_s)$ giving an atom of $\mathcal{R}_{q,k}^m$ of positive $\xi_{k,n}^{M,m}$-measure satisfies $\bigcap\limits_{s=0}^{m-1} h_{k,s}^{-1}(I_s \times J_s) \subset X$ and $\bigcap\limits_{s=0}^{m-1} h_s^{-1}(I_s \times J_s) \subset X$
\end{itemize}
Then each atom of $\mathcal{R}_{q,k}^m$ is of the form $\bigcap\limits_{s=0}^{m-1} f_k^{-s}(I_s \cap (\log |f_k'|)^{-1}(J_s)) = \bigcap\limits_{s=0}^{m-1} h_{k,s}^{-1}(I_s \times J_s)$ with $J_s \in I_q$ and $I_s \in \mathcal{I}(\varepsilon_q)$.
Therefore, by the choice of $a$, of $\mathcal{I}(\varepsilon_q)$ and by Lemma \ref{lemma:cv-multi-unif-mes}, the $\xi_k^{M,m}$-measure of such atoms converges, as $k$ goes to infinity, to $\xi^{M,m}\left ( \bigcap\limits_{s=0}^{m-1} f^{-s}(I_s \cap (\log |f'|)^{-1}(J_s)) \right )$.
\end{proof}

We can now prove the main result of this section.

\begin{proof}[Proof of Proposition \ref{prop:entropy-cv}]
Recall that we want to prove
$$
\lim\limits_{m \to +\infty} \lim\limits_{M \to +\infty} \lim\limits_{k \to +\infty} \frac{1}{m} H_{\overline{\xi_k^{M,m}}}(\mathcal{R}_{q,k}^m) = h(\mu_1, \mathcal{R}_{q}) \leq h(\mu_1).
$$
From Lemma \ref{lemma:CV-k-entropy}, we have $H_{\xi_k^{M,m}}(\mathcal{R}_{q,k}^m) \underset{k \to +\infty}{\longrightarrow} H_{\xi^{M,m}}(\mathcal{R}_{q}^m)$.
Notice as well that
$$
H_{\overline{\xi_k^{M,m}}}(\mathcal{R}_{q,k}^m)
= \log \beta_k^{M,m} + \frac{1}{\beta_k^{M,m}} H_{\xi_k^{M,m}}(\mathcal{R}_{q,k}^{M,m}).
$$
Then, by using Lemma \ref{lemma:monotone-xi} and by following the second part of the proof of Lemma 7.2 from \cite{A}, we obtain
$$
H_{\overline{\xi^{M,m}}}(\mathcal{R}_{q}^m) \underset{M \to +\infty}{\longrightarrow} H_{\mu_1}(\mathcal{R}_q^m).
$$
We conclude by using item $ii)$ from Lemma \ref{lemma:entropy-R}, which gives
$$
\lim\limits_{m \to +\infty}\frac{1}{m} H_{\mu_1}(\mathcal{R}_q^m) = h(\mu_1, \mathcal{R}_q) \leq h(\mu_1).
$$
\end{proof}

\subsection{Continuity of the exponent of the hyperbolic component}
\label{ssec:cont-expo-HB}

We follow the notations from section \ref{ssec:hyperbolic-times} and we let $\mathtt{F}_k = \mathtt{F}_{\nu_k}$ and we denote by $m_{k,n}(x)$ (resp. $m_{k,n}^+(x)$) the $m_n(x)$ (resp. $m_n^+(x)$) for $f_k$.
Hence the convergence $\frac{m_{k,n}^+(x)}{n} \to 1$ is uniform on $\mathtt{F}_k$. We also defined the measure $\zeta_k$ to be the normalized restriction of $\nu_k$ to $\mathtt{F}_k$.
We define $\xi_{k,n,\mathtt{F}_k}^M$ and $\eta_{k,n,\mathtt{F}_k}^M$ the same way as $\xi_{k,n}^M$ and $\eta_{k,n}^M$ in section \ref{ssec:hyperbolic-times}, but by pushing forward $\zeta_k$ instead of $\nu_k$.
When $n$ goes to infinity, we also define $\xi_{k,\mathtt{F}_k}^M$ (resp. $\eta_{k,\mathtt{F}_k}^M$) to be the limit of $\xi_{k,n,\mathtt{F}_k}^M$ (resp. $\eta_{k,n,\mathtt{F}_k}^M$).\\

For each $k$, the map $\log |f_k'|$ is $\nu_k$-integrable.
Therefore, we may assume that $\nu_k(\mathtt{F}_k)$ converges quickly enough to 1 so that
$$
\lim\limits_{k \to +\infty} |\lambda_{f_k}(\nu_k) - \lambda_{f_k}(\zeta_k)| = 0,
$$
Then, we have $\zeta_k = \xi_{k,n,\mathtt{F}_k}^M + \eta_{k,n,\mathtt{F}_k}^M$, so that
$$
\lambda_{f_k}(\zeta_k) = \lambda_{f_k}(\xi_{k,n,\mathtt{F}_k}^M) + \lambda_{f_k}(\eta_{k,n,\mathtt{F}_k}^M),
$$
where $\lambda_{f_k}(\xi_{k,n,\mathtt{F}_k}^M) = \int \log |f_k'| d\xi_{k,n,\mathtt{F}_k}^M$ and $\lambda_{f_k}(\eta_{k,n,\mathtt{F}_k}^M) = \int \log |f_k'| d\eta_{k,n,\mathtt{F}_k}^M$.
In this section, we prove that $\lambda_{f_k}(\xi_{k,n,\mathtt{F}_k}^M)$ converges to $\beta \lambda_f(\mu_1)$.

\begin{proposition}
\label{proposition:mu1-integrable}
The function $\log |f'|$ is $\mu_1$-integrable.
\end{proposition}

\begin{proof}

Since $\log |f'| \leq \log || f' ||_{\infty}$, we only have to prove that $\int \log^- |f'| d \mu < + \infty$. For $m \in \bbN$, let
$$
h_{k,m} = \min\{m, \log^- |f_k'|\} : I \to \bbR_+ \underset{m \to +\infty}{\nearrow} \log^-|f_k'|,
$$
and
$$
h_m = \min\{m, \log^- |f'|\} : I \to \bbR_+ \underset{m \to +\infty}{\nearrow} \log^-|f'|.
$$
For any fixed $m$, by using that $f_k' \to f'$ uniformly on $I$, we obtain that $h_{k,m} \to h_m$ uniformly on $I$. Hence, by continuity of $h_{k,m}$ and $h_m$ on $I$ and by monotone convergence, we have 
$$
\int \log^- |f'| d \mu_1 = \lim\limits_{m \to +\infty} \int h_m d \mu_1 = \lim\limits_{m \to +\infty}\lim\limits_{M \to +\infty}\lim\limits_{k \to +\infty}\lim\limits_{n \to +\infty}\int h_{k,m} d \overline{\xi_{k,n}^M}.
$$
We then conclude by the same argument as in Proposition 5.3 from \cite{A}.
\end{proof}

\begin{proposition}
\label{prop:conv-expo-HB}
We have
\begin{align*}
\int \log |f_k'| d \xi_{k,n,\mathtt{F}_k}^{M} &\underset{n \to +\infty}{\longrightarrow}
\int \log |f_k'| d \xi_{k,\mathtt{F}_k}^{M}\\
&\underset{k \to +\infty}{\longrightarrow}\int \log |f'| d \xi^{M}\\
&\underset{M \to +\infty}{\longrightarrow} \beta \int \log |f'| d \mu_1 = \beta \lambda_f(\mu_1).
\end{align*}
\end{proposition}

\begin{proof}
For $k,M$ define the sets $A_k^M = \{x \in I : \log |f_k'(x)| \geq - M \log || f_k' ||_{\infty} \}$ and $A^M = \{x \in I : \ \log |f'(x)| \geq - M \log ||f' ||_{\infty} \}$.
Thus, from Lemma \ref{lemma:prop-hb}, the support of $\xi_{k,n,\mathtt{F}_k}^{M}$ and of $\xi_{k,\mathtt{F}_k}^{M}$ are in $A_k^M$.\\

Limit in $n$:
Since $\log |f_k'|$ is continuous on $A_k^M$, we have
$$
\int \log |f_k'| d \xi_{k,n,\mathtt{F}_k}^{M} = \int_{A_k^M} \log |f_k'| d \xi_{k,n,\mathtt{F}_k}^{M}
\underset{n \to +\infty}{\longrightarrow} \int_{A_k^M} \log |f_k'| d \xi_{k,\mathtt{F}_k}^{M} = \int \log |f_k'| d \xi_{k,\mathtt{F}_k}^{M}.
$$

Limit in $k$:
Since $f_k \to f$ $\mathcal{C}^r$-weakly, we have that $f_k' \to f'$ uniformly on $I$.
Let $\eta > 0$.
By the same argument as in Lemma \ref{lemma:CV-k-entropy}, one can choose $k$ large enough such that $A_k^M \subset A^{2M}$ and $|\log |f_k'| - \log |f'|| \leq \eta$ on $A_k^M$. Thus
$$
\left | \int \log |f_k'| d \xi_{k,\mathtt{F}_k}^M - \int \log |f'| d \xi_{k,\mathtt{F}_k}^M \right |
\leq \int_{A_k^M} \left | \log |f_k'| - \log |f'|\right | d \xi_{k,\mathtt{F}_k}^M
\leq \eta.
$$
Then, for every $k$, the support of $\xi_k^M$ and $\xi^M$ is inside $A^{2M}$, where $\log |f'|$ is continuous. Hence it remains to show that $(\xi_{k,\mathtt{F}_k}^M)_k$ converges to $\xi^M$ in the weak-$*$ topology.
Let $\varphi : I \to \bbR$ be a continuous map.
We have
$$
\left | \int \varphi d \xi_{k,\mathtt{F}_k}^M \times \nu_k(\mathtt{F}_k) - \int \varphi d \xi_{k}^M \right |
\leq \nu_k(I \backslash \mathtt{F}_k) || \varphi ||_{\infty}.
$$
Since we chose $\nu_k(\mathtt{F}_k) \underset{k \to +\infty}{\longrightarrow} 1$, the convergence of $(\xi_{k}^M)_k$ to $\xi^M$ in the weak-$*$ topology gives
$$
\int \log |f'| d \xi_{k,\mathtt{F}_k}^M \underset{k \to +\infty}{\longrightarrow} \int \log |f'| d \xi^M.
$$

Limit in $M$:
Notice that Lemma \ref{lemma:monotone-xi} would conclude if $\log |f'|$ were to be a characteristic function.
Hence, by linearity, monotone convergence, and using the fact that $(\xi^M)_M$ is non-decreasing, it is still true for non-negative measurable functions. Therefore, it is true for any function that is $\mu_1$-integrable. 
We thus get the result using Proposition \ref{proposition:mu1-integrable}.\\

Last equality:
The measure $\mu_1$ is $f$-invariant and $\log |f'|$ is $\mu_1$-integrable from Proposition \ref{proposition:mu1-integrable}, so Birkhoff's ergodic theorem concludes.
\end{proof}

\section{Proof of Theorem \ref{th:main1}}
\label{sc:Proof}

\subsection{Estimating the exponent of the neutral component}
\label{sec:expo-neutral}

As announced in the beginning of section \ref{ssec:cont-expo-HB}, we now estimate the quantity $\lambda_{f_k}(\eta_{k,n,\mathtt{F}_k}^M)$.

\begin{proposition}
\label{prop:expo-neutral}
For any $k \in \bbN$ and $M \in \bbN^*$, we have
$$
\liminf\limits_{n \to +\infty} \int \log |f_k'| d \eta_{k,n,\mathtt{F}_k}^M \geq (1 - \beta_k^M) \delta.
$$
\end{proposition}

\begin{proof}
For $k \in \bbN$ and $M \in \bbN^*$, we have
\begin{align*}
\nu_k(\mathtt{F}_k) \int \log |f_k'| d \eta_{k,n,\mathtt{F}_k}^M
= \frac{1}{n} \int_{\mathtt{F}_k} \sum\limits_{i \in [\![0 ; n [\![ \backslash E_{k,n}^M(x)} \log |f_k'(f_k^i x)| d \nu_k(x).
\end{align*}
We write $E_{k,n}^M(x) = \bigcup\limits_{j=1}^l [\![ a_j ; b_j [\![$. 
Thus, by noting $b_0 = 0$, we have
\begin{align*}
\hspace{3em}&\hspace{-3em}\sum\limits_{i \in [\![ 0 ; n [\![ \backslash E_{k,n}^M(x)} \log |f_k'(f_k^i x)|\\
&= \sum\limits_{j=0}^{l-1} \log |(f_k^{a_{j+1} - b_j})'(f_k^{b_j} x)| + \log |(f_k^{n-b_l})'(f_k^{b_l} x)|\\
&= \sum\limits_{j=0}^{l-1} \log |(f_k^{a_{j+1} - b_j})'(f^{b_j} x)| + \log |(f_k^{m_{k,n}^+(x)-b_l})'(f_k^{b_l} x)| - \log |(f_k^{m_{k,n}^+(x) - n})'(f_k^{n} x)|
\end{align*}
Then from the fact that $m_n^+(x) \in E(x)$ and $a_{j+1} \in E(x)$ for $j \in [\![0;l-1]\!]$, we have
\begin{align*}
\hspace{3em}&\hspace{-3em}\sum\limits_{i \in [\![ 0 ; n [\![ \backslash E_{k,n}^M(x)} \log |f_k'(f_k^i x)|\\
&\geq \left ( \sum\limits_{j=0}^{l-1} \delta(a_{j+1} - b_j) \right )+ \delta (m_{k,n}^+(x) - b_l) - (m_{k,n}^+(x) - n) \log || f_k'||_{\infty}\\
&= (n - \sharp E_{k,n}^M(x)) \delta - (m_{k,n}^+(x) - n) (\log || f_k'||_{\infty} - \delta).
\end{align*}
By integrating, we obtain
\begin{align*}
\nu_k(\mathtt{F}_k) \int \log |f_k'| d \eta_{k,n,\mathtt{F}_k}^M
&\geq (1- \beta_{k,n}^M) \delta - \nu_k(I \backslash \mathtt{F}_k) \delta - \int_{\mathtt{F}_k} \left ( \frac{m_{k,n}^+(x)}{n} - 1\right ) (\log || f_k'||_{\infty} - \delta) d \nu_k(x).
\end{align*}
Then $\frac{m_{k,n}^+(x)}{n} \underset{n \to +\infty}{\longrightarrow} 1$ uniformly on $\mathtt{F}_k$ and $\nu_k(I\backslash \mathtt{F}_k) \to 0$ conclude.
\end{proof}

\subsection{Proof of the Main Proposition}
\label{ssec:proof}

We assume hypotheses \ref{itm:main-prop1} and \ref{itm:main-prop2} from Proposition \ref{prop:main-prop}.
We show that the definitions at the beginning of section \ref{sec:main-prop} of the measures $\mu_0$ and $\mu_1$ and of $\beta$ satisfy the conclusion of Proposition \ref{prop:main-prop}. Recall that the $f$-invariance of $\mu_0$ and $\mu_1$ has been proved in Lemma \ref{lemma:f-inv}.
We start by proving item \ref{itm:main21'} from Theorem \ref{th:main2}, which also implies item \ref{itm:main11} from Theorem \ref{th:main1}. 

\begin{lemma}
\label{lemma:mu1-HB}
For $\mu_1$-almost every $x \in I$, we have 
$\lambda_f(x) \geq \delta$.
\end{lemma}

\begin{proof}
We adapt the proof of Lemma 10 from \cite{B}.
Notice that is suffices to prove that for any $\varepsilon > 0$, for $\mu_1$-almost every $x \in I$, $\lambda_f(x) \geq (1-\varepsilon) \delta$.
Hence, let $\varepsilon > 0$.
For $M \in \bbN^*$ and $k \in \bbN$, let $F_M^k := \{ x \in I : \exists m \in [\![ 1 ; M ]\!], \log |(f_k^m)'(x)| \geq m \delta \}$.
With the same argument as in Proposition 5.4 from \cite{A}, we can show that $\overline{\xi_{k,n}^M}(F_M^k) = 1$, where $\overline{\xi_{k,n}^M}$ is the normalization of $\xi_{k,n}^M$.
We also define $F_M := \{ x \in I : \exists m \in [\![1 ; M ]\!], \log |(f^m)'(x)| \geq m (1 - \varepsilon) \delta \}$.
Hence, for $k$ large enough depending on $M$, we have $F_M^k \subset F_M$. This is why we took $(1-\varepsilon) \delta$ instead of $\delta$.
Then, we define $F = \bigcup\limits_{M > 1} F_M$, and with the same argument as in Proposition 5.4 from \cite{A}, we can prove that $\beta \mu_1(F) = \beta$, thus $\mu_1(F) = 1$ because Remark \ref{rmk:density-E} implies $\beta > 0$, as noted in section \ref{ssec:main-prop}.
We then conclude the same way as in Proposition 5.4 from \cite{A}.
\end{proof}

Before proving item \ref{itm:main12} from Theorem \ref{th:main1}, we prove the following lemma:

\begin{lemma}
\label{lemma:term-II}
We have
$$
0 = \lim\limits_{q \to +\infty} \lim\limits_{m \to +\infty} \lim\limits_{M \to +\infty} \lim\limits_{k \to +\infty} \lim\limits_{n \to +\infty}
\frac{1}{q}\int \log^+ |(f_k^q)'| d \xi_{k,n}^{M,m,q} - \frac{1}{q}\int \log |(f_k^q)'| d \xi_{k,n}^{M,m,q}.
$$
\end{lemma}

\begin{proof}
By applying the same argument as in Proposition \ref{prop:conv-expo-HB} to $\log |(f_k^q)'|$ for the measure $\xi_{k,n}^{M,m,q}$, we obtain that it suffices to show
$$
0 = \lim\limits_{q \to +\infty} \frac{1}{q} \int \log^+ |(f^q)'| d \mu_1 - \frac{1}{q} \int \log |(f^q)'| d \mu_1.
$$
Note that taking $\xi_{k,n}^{M,m,q}$ instead of $\xi_{k,n}^{M}$ is useful to control $\log |(f_k^q)'|$, although taking $m \geq q$ also suffices.
Then, we conclude by using Kingman's ergodic theorem and the positivity of the Lyapunov exponent $\mu_1$-a.e. from Lemma \ref{lemma:mu1-HB}.
\end{proof}

We now prove item \ref{itm:main12} from Theorem \ref{th:main1}.

\begin{proposition}
We have
$$
\limsup\limits_{k \to +\infty} h_{f_k}(\nu_k) \leq \beta h_f(\mu_1) + (1-\beta) \alpha.
$$
\end{proposition}

\begin{proof}
Let $k \in \mathcal{K}$, which is a sequence along which $(h_{f_k}(\nu_k))_k$ converges to its limsup. We start with some notations.
Let $\iota_k > 0$ such that $\iota_k \to 0$.
Let $\mathcal{P}_k$ be a finite partition such that $\nu_k(\partial \mathcal{P}_k) = 0$ and such that
$$
\lim\limits_{n \to +\infty} \frac{1}{n} H_{\nu_k}(\mathcal{P}_k^n) \geq h(\nu_k) - \iota_k.
$$
Choose $\mathtt{F}_k$ as in section \ref{ssec:hyperbolic-times}, and assume that its $\nu_k$-measure is sufficiently close to $1$ so that
$$
\nu_k(I \backslash \mathtt{F}_k) \log \sharp\mathcal{P}_k \underset{k \to +\infty}{\longrightarrow} 0.
$$
Indeed, the sequence $(\sharp \mathcal{P}_k)_k$ may not be bounded.
Let $M, m \in \bbN^*$.
For $q \in \bbN^*$, we consider the partition $\mathcal{R}_{q,k}$ from section \ref{sec:entropy-HB}.
From Lemma \ref{lemma:entropy-R}, the partition $\mathcal{R}_{q,k}$ satisfies the hypotheses of Proposition \ref{prop:entropy-ineq}.
Therefore, from Proposition \ref{prop:entropy-final}, we have for every $m \in \bbN^*$
\begin{align*}
\limsup\limits_{k \to +\infty} h_{f_k}(\nu_k)
&\leq \limsup\limits_{k \to +\infty} \beta_k^{M,m} \frac{1}{m} H_{\overline{\xi_k^{M,m}}}(\mathcal{R}_{q,k}^m) + \limsup\limits_{n \to +\infty} \frac{1}{n} H_{\zeta_k}(\mathcal{P}_k^n \mid \mathcal{R}_{q,k}^{E_{k,n}^{M,m}})\\
&+ \nu(I \backslash \mathtt{F}_k) \log \sharp \mathcal{P}_k + \frac{3 \log M}{M} + \iota_k  \\
&+\limsup\limits_{n \to +\infty} \frac{m}{n}\sum\limits_{E \in \mathcal{E}_{k,n}^{M,m}} \nu_k(E) \log \left ( \sharp(\mathcal{R}_{q,k})_{\nu_k^E} \right ) \sharp\partial E. 
\end{align*}

Then, by letting $M$ then $m$ go to infinity and applying Proposition \ref{prop:entropy-cv} and item \ref{itm:entropy-R6} from Lemma \ref{lemma:entropy-R}, we obtain
\begin{align*}
\limsup\limits_{k \to +\infty} h_{f_k}(\nu_k) \leq \beta h_f(\mu_1) + \limsup\limits_{q \to +\infty}\limsup\limits_{m \to +\infty} \limsup\limits_{M \to +\infty} \limsup\limits_{k \to +\infty} \limsup\limits_{n \to +\infty} \frac{1}{n} H_{\zeta_k}(\mathcal{P}_k^n \mid \mathcal{R}_{q,k}^{E_{k,n}^M}).
\end{align*}

By Lemma \ref{le:reparamization-partition} and Lemma \ref{lemma:repabound}, we have
$$
\begin{aligned}
\limsup\limits_{n \to +\infty} \frac{1}{n}H_{\zeta_k}(\mathcal{P}_k^n \mid \mathcal{R}_{q,k}^{E_{k,n}^{M,m}})&\leq \limsup\limits_{n \to +\infty} \sum_{ Q\in \mathcal{R}^{E^{M,m}_n}_{q,k}} \zeta_k(Q)\frac{1}{n}\log\sharp\Psi_{Q}\\
&\leq \limsup\limits_{n \to +\infty} \int \left( 1-\frac{\sharp 
E_{k,n}^{M,m}(x)}{n} \right)C(f_k) d \zeta_k\\
&+\frac{1}{(r-1)\cdot \nu_k(\mathtt{F}_k)}\left( \int \frac{\log^{+}|(f_k^q)'|}{q}d\xi^{M,m,q}_{k,n}-\int \frac{\log|(f_k^q)'|}{q}d\xi^{M,m,q}_{k,n} \right)\\
&+\gamma_{n,k,M,m,q}(f_k).
\end{aligned}
$$
We then investigate the three terms separately:
$$
\begin{aligned}
    &\operatorname{(I)}= \int \left( 1-\frac{\sharp E^{M,m}_{k,n}(x)}{n} \right)C(f_k) d \zeta_k;\\
    &\operatorname{(II)}=\frac{1}{(r-1)\cdot \nu_k(\mathtt{F}_k)}\left( \int \frac{\log^{+}|(f_k^q)'|}{q}d\xi^{M,m,q}_{k,n}-\int \frac{\log|(f_k^q)'|}{q}d\xi^{M,m,q}_{k,n} \right);\\
    &\operatorname{(III)}=\gamma_{n,k,M,m,q}(f_k).
\end{aligned}
$$

Firstly, we have $\int \frac{\sharp E^{M,m}_{k,n}(x)}{n} d \zeta_k \rightarrow \beta$ when $m,M,k,n\rightarrow +\infty$. Thus by continuity of $f \mapsto C(f)$ in the $\mathcal{C}^1$-topology, we have
$$
\limsup\limits_{m \to +\infty} \limsup\limits_{M \to +\infty} \limsup\limits_{k \to +\infty} \limsup\limits_{n \to +\infty}\operatorname{(I)}=(1-\beta) C(f).
$$

Secondly, by Lemma \ref{lemma:term-II},
$$
\limsup\limits_{q \to +\infty}\limsup\limits_{m \to +\infty} \limsup\limits_{M \to +\infty} \limsup\limits_{k \to +\infty} \limsup\limits_{n \to +\infty}\operatorname{(II)}=0.
$$

Thirdly, by Lemma \ref{lemma:repabound}, we have
$$
\limsup\limits_{q \to +\infty}\limsup_{m\rightarrow +\infty}\limsup_{M\rightarrow +\infty}\limsup_{k\rightarrow+\infty}\limsup_{n\rightarrow +\infty}\operatorname{(III)}=0
$$
Therefore, hypothesis \ref{itm:main-prop2} from Proposition \ref{prop:main-prop} gives
$$
\limsup\limits_{k \to +\infty} h_{f_k}(\nu_k) \leq \beta h_f(\mu_1) + (1-\beta) \alpha.
$$
\end{proof}

We are left to prove item \ref{itm:main23'} from Theorem \ref{th:main2} for $\delta$ defined with $p=1$.

\begin{proposition}
We have
$$
\liminf\limits_{\mathcal{K} \ni k \to +\infty} \lambda_{f_k}(\nu_k) \geq \beta \lambda_f(\mu_1) + (1 - \beta) \delta.
$$
\end{proposition}

\begin{proof}
Recall that in the beginning of section \ref{ssec:cont-expo-HB}, we explained that we can replace $\nu_k$ with $\zeta_k$ and we wrote
$$
\lambda_{f_k}(\zeta_k) = \lambda_{f_k}(\xi_{k,n,\mathtt{F}_k}^M) + \lambda_{f_k}(\eta_{k,n,\mathtt{F}_k}^M).
$$
In that section, we also proved Proposition \ref{prop:conv-expo-HB}, which gives
$$
\lim\limits_{M \to +\infty} \lim\limits_{k \to +\infty} \lim\limits_{n \to +\infty} \lambda_{f_k}(\xi_{k,n,\mathtt{F}_k}^M) = \beta \lambda_f(\mu_1).
$$
Then, Proposition \ref{prop:expo-neutral} gives
$$
\liminf\limits_{M \to +\infty} \liminf\limits_{k \to +\infty} \liminf\limits_{n \to +\infty} \lambda_{f_k}(\eta_{k,n,\mathtt{F}_k}^M) \geq (1 - \beta) \delta.
$$
Recall that $M,k,n$ actually belong to some subsequences given in section \ref{ssec:main-prop} by Cantor's diagonal argument.
In particular, $k$ may not go through all of $\mathcal{K}$.
However, we can obtain the liminf in $k \in \mathcal{K}$ in the statement by first choosing a subsequence of $\mathcal{K}$ along which $(\lambda_{f_k}(\nu_k))$ converges to its liminf, then use Cantor's diagonal argument again.  
This concludes the proof of item $3')$ of Theorem \ref{th:main2}, hence of Proposition \ref{prop:main-prop}.
\end{proof}

\subsection{Proof of Theorems \ref{th:main1} and \ref{th:main2}}
\label{section:proof-th}

We show how to obtain an iterate of $f$ which satisfies the hypotheses of Proposition \ref{prop:main-prop}.
We first deal with the cases where Theorems \ref{th:main1} and \ref{th:main2} are true for $\beta = 0$.
Recall that $\mathcal{K}$ is a sequence of integers such that $(h_{f_k}(\nu_k))_k$ converges to its limsup along $\mathcal{K}$.

\begin{lemma}
\label{th:small-liminf}
If $\liminf\limits_{\mathcal{K} \ni k \to +\infty} \lambda_{f_k}(\nu_k) \leq \alpha$, then Theorems \ref{th:main1} and \ref{th:main2} work with $\beta = 0$.
\end{lemma}

\begin{proof}
We let $\mu_1$ be any hyperbolic measure and $\mu_0 = \mu$. Then Ruelle's inequality \cite{Ruelle1978} gives $\limsup\limits_{k \to +\infty} h_{f_k}(\nu_k) = \lim\limits_{k \in \mathcal{K}} h_{f_k}(\nu_k) \leq \alpha$.
\end{proof}

We first give some notations to deal with iterates of $f_k$ and $f$.
If $p \geq 1$, we will consider $f_k^p$ and $f^p$.
Then, we want a converging sequence of $f_k^p$-ergodic measures.
Let $\nu_k^p$ be an ergodic component of $\nu_k$ for $f^p$, thus
$$
h_{f_k^p}(\nu_k^p) = p h_{f_k}(\nu_k) \; \; \; \text{and} \; \; \; \lambda_{f_k^p}(\nu_k^p) = p \lambda_{f_k}(\nu_k).
$$
By taking a subsequence, contained in $\mathcal{K}$, we may assume that $(\nu_k^p)_k$ converges to some $\mu^p$.
Notice that $\mu = \frac{1}{p}\sum\limits_{i=0}^{p-1} f^i_* \mu^p$.
Let $\alpha_p = p \alpha > \frac{p R(f)}{r} = \frac{R(f^p)}{r}$.

\begin{lemma}
\label{lemma:main-prop-iterate}
If
$\liminf\limits_{\mathcal{K} \ni k \to +\infty} \lambda_{f_k}(\nu_k) > \alpha$, then there exists $p$ such that $f_k^p$, $f^p$, $\nu_k^p$ and $\mu^p$ satisfy the hypotheses of Proposition \ref{prop:main-prop}. 
\end{lemma}

\begin{proof}
There is $p$ large enough such that
$$
\liminf\limits_{\mathcal{K} \ni k \to +\infty} \lambda_{f_k}(\nu_k) > \alpha - \frac{R(f)}{r} + \frac{1}{rp} \log^+ ||(f^p)'||_{\infty}.
$$
Therefore, we have
$$
\liminf\limits_{\mathcal{K} \ni k \to +\infty} \lambda_{f_k^p}(\nu_k^p) = p\liminf\limits_{\mathcal{K} \ni k \to +\infty} \lambda_{f_k}(\nu_k) > \alpha_p - \frac{R(f^p)}{r} + \frac{1}{r} \log^+ ||(f^p)'||_{\infty}.
$$
which is hypothesis \ref{itm:main-prop1} from Proposition \ref{prop:main-prop}.
For item \ref{itm:main-prop2}, notice that $C(f^p)/p \underset{p \to +\infty}{\longrightarrow} \frac{R(f)}{r} < \alpha$.
Therefore, we can assume that $p$ is large enough so that $C(f^p) < p \alpha = \alpha_p$, which is hypothesis \ref{itm:main-prop2} from Proposition \ref{prop:main-prop}.
\end{proof}

We are now ready to prove Theorems \ref{th:main1} and \ref{th:main2}.

\begin{proof}[Proof of Theorems \ref{th:main1} and \ref{th:main2}]
If
$\liminf\limits_{\mathcal{K} \ni k \to +\infty} \lambda_{f_k}(\nu_k) \leq \alpha$,
then Lemma \ref{th:small-liminf} concludes.
Assume now that
$\liminf\limits_{\mathcal{K} \ni k \to +\infty} \lambda_{f_k}(\nu_k) > \alpha$.
Let $p$ be given by Lemma \ref{lemma:main-prop-iterate}.
We apply Proposition \ref{prop:main-prop} to the $p$-th iterate of $f$, which gives $\beta_p > 0, \mu_0^p$ and $\mu_1^p$ such that Theorems \ref{th:main1} and \ref{th:main2} hold for $f_k^p, f^p, \nu_k^p,\mu^p$ and $\delta_p = \min \left \{ \alpha_p - \frac{R(f^p)}{r}, \frac{r-1}{r} \log (16/9) \right \}$.
Let $\beta = \beta_p$, $\mu_{0/1} = \frac{1}{p} \sum\limits_{i=0}^{p-1} f^i_* \mu_{0/1}^p$, $\delta = \frac{1}{p} \delta_p$.
We show that these quantities give Theorems \ref{th:main1} and \ref{th:main2} for $f,f_k,\nu_k$ and $\mu$.\\

Proof of \ref{th:main2}.\ref{itm:main21}:
We chose $\beta = 0$ whenever the announced condition was not satisfied. Conversely, if this condition is satisfied, then we apply Proposition \ref{prop:main-prop} which gives $\beta_p > 0$, so $\beta = \beta_p$ concludes.\\

Proof of \ref{th:main2}.3)\ref{itm:main21'}:
We have a set of full $\mu_1^p$-measure where $\lambda_{f^p} \geq \delta_p = p \delta$.
Therefore, we have a set of full $\mu_1^p$-measure where $\lambda_f \geq \delta$.
Since this property is $f$-invariant $\mu_1^p$-a.e., there is a set of full $\mu_1$-measure where $\lambda_f \geq \delta$, which concludes.\\

Proof of \ref{th:main1}.\ref{itm:main11}:
This is a consequence of \ref{th:main2}.3)1').\\

Proof of \ref{th:main2}.3)\ref{itm:main23'}:
We have
$$
\liminf\limits_{k \to +\infty} \lambda_{f_k}(\nu_k) = \frac{1}{p}\liminf\limits_{k \to +\infty} \lambda_{f_k^p}(\nu_k^p)
\geq \frac{1}{p} \left ( \beta_p \lambda_{f^p}(\mu_1^p) + (1-\beta_p) \delta_p\right )
=\beta \lambda_f(\mu_1) + (1-\beta) \delta.
$$

Proof of \ref{th:main1}.\ref{itm:main13}:
This is a consequence of \ref{th:main2}.3)3').\\

Proof of \ref{th:main1}.\ref{itm:main12}:
We have
$$
\limsup\limits_{k \to +\infty} h_{f_k}(\nu_k) = \frac{1}{p} \limsup\limits_{k \to +\infty} h_{f_k^p}(\nu_k^p)
\leq \frac{1}{p}\left ( \beta_p h_{f^p}(\mu_1^p) + (1 - \beta_p) \alpha_p \right )
= \beta h_f(\mu_1) + (1-\beta)\alpha.
$$
\end{proof}

\begin{proof}[Proof of Proposition \ref{prop:Cinf-case}]
We show that when $r=\infty$, then taking $\alpha = 0$ works.
We let $\beta^{\alpha}, \mu_1^{\alpha}$ and $\mu_0^{\alpha}$ be given by Theorem \ref{th:main1}.
We notice that for $\alpha \leq \alpha'$, we have $\delta^{\alpha} \leq \delta^{\alpha'}$, hence $E_k(x)^{\alpha'} \subset E_k(x)^{\alpha}$, and $(\xi_{k,n}^{M,m,q})^{\alpha'} \leq (\xi_{k,n}^{M,m,q})^{\alpha}$.
By taking limits, we get $\beta^{\alpha'} \leq \beta^{\alpha}$ and $\beta^{\alpha'} \mu_1^{\alpha'} \leq \beta^{\alpha} \mu_1^{\alpha}$.
Then the same argument as for Lemma \ref{lemma:monotone-xi} gives that for any Borel set $B$, we have $\beta^{\alpha'} \mu_1^{\alpha'}(B) \leq \beta^{\alpha} \mu_1^{\alpha}(B)$.
Therefore, if we let $\mu_1, \mu_0$ and $\beta$ be the limits as $\alpha$ goes to zero of $\mu_1^{\alpha}, \mu_0^{\alpha}$ and $\beta^{\alpha}$ respectively, then by the same argument as in Proposition \ref{prop:conv-expo-HB}, we get for any $\mu_1$-integrable map $\psi : I \to \bbR$ that $\beta^{\alpha}\mu_1^{\alpha}(\psi) \underset{\alpha \to 0}{\longrightarrow} \beta \mu_1(\psi)$.
This gives items $1)$ and $3)$ from Theorem \ref{th:main1}. Item $2)$ is a consequence of the upper semi-continuity of the metric entropy for $\mathcal{C}^{\infty}$ maps, as proved by Newhouse in \cite{NewhouseContEntropy}.
\end{proof}

\section{Uniform integrability and Theorem \ref{th:main3}}
\label{sc:UI}

Throughout this section, we always consider $(f_k)_{k\in \mathbb{N}}$ as a sequence of $\mathcal{C}^{r}$ maps of $I$ converging in the $\mathcal{C}^{r}$-topology to $f$.

\begin{definition}
    Let $(\nu_k)_{k\in \mathbb{N}}$ be a convergent sequence of measures, $\nu_k\xrightarrow[]{\ast} \mu$, with each $\nu_k\in \mathbb{P}_{\rm{erg}}(f_k)$ and such that $(h_{f_k}(\mu_k))_{k\in \mathbb{N}}$ converges, the \emph{uniform integrability defect} of $(f_k, \nu_k)_{k\in\mathbb{N}}$ is
 $$
     \gamma:=\lim_{\eta\to0} \limsup_{k\to\infty} \int_{\{x:|f_k'(x)|<\eta\}} - \log|f'_k|\, d\nu_k.
 $$ 
\end{definition}

We give an alternative expression of the quantity $\gamma$.

\begin{lemma}\label{defcon}
The uniform integrability defect $\gamma$
of a sequence of measures gives the defect in lower semicontinuity of the Lyapunov exponent:
 $$
       \gamma=\lambda_{f}(\mu)-\liminf_{k\to\infty} \lambda_{f_k}(\nu_k).
 $$
\end{lemma}

Since Lyapunov exponents are already upper semi-continuous, from Lemma \ref{defcon} one sees that $\gamma=0$ if and only if
$$
\lim_{k\rightarrow \infty}\lambda_{f_k}(\mu_k)=\lambda_f(\mu).
$$


\begin{theorem}\label{th:UI}
Let $f_k\rightarrow f$ be a sequence of $\mathcal{C}^{r}$ maps converging in the $\mathcal{C}^{r}$-topology. Let $(\mu_k)_{k\in \mathbb{N}}$ be a sequence of measures, with $\mu_k\in \mathbb{P}_{\rm{erg}}(f)$, $\mu_k\xrightarrow[]{\ast}\mu$, and $(h_{f_k}(\mu_k))_k$ converging. If $h_{\text{top}}(f) > \frac{R(f)}{r}$, then
$$
\lim_{k\rightarrow \infty} h_{f_k}(\mu_k)\leq (1-\frac{\gamma}{\log||f' ||_{\infty}})h_{\rm{top}}(f) + \frac{\gamma}{\log ||f' ||_{\infty}} \frac{R(f)}{r}.
$$
\end{theorem}

\begin{proof}
We have
\begin{align*}
\gamma
&=\lambda_{f}(\mu)-\liminf_{\mathcal{K} \ni k\to\infty} \lambda_{f_k}(\nu_k)\\
\text{Theorem \ref{th:main2}.2)3')}\;&\leq \lambda_f(\mu) - \beta \lambda_f(\mu_1) - (1-\beta)\delta\\
&\leq (1-\beta)\lambda_f(\mu_0)\\
&\leq (1-\beta)\log ||f' ||_{\infty}.
\end{align*}
By Theorem \ref{th:main1}.2) we have
\begin{align*}
\lim\limits_{k \to +\infty} h_{f_k}(\nu_k)
&\leq \beta h_f(\mu_1) + (1-\beta) \alpha
\leq \beta (h_{\text{top}}(f) - \alpha) + \alpha.
\end{align*}
Since $h_{\text{top}}(f) > \frac{R(f)}{r}$, we can use the previous upper bound of $\gamma$ and let $\alpha$ go to $\frac{R(f)}{r}$, which gives
\begin{align*}
\lim\limits_{k \to +\infty} h_{f_k}(\nu_k)
&\leq (1 - \frac{\gamma}{\log ||f' ||_{\infty}}) h_{\rm{top}}(f) + \frac{\gamma}{\log ||f' ||_{\infty}} \frac{R(f)}{r}.
\end{align*}
\end{proof}

We now prove that Theorem \ref{th:UI} implies Theorem \ref{th:main3}:

\begin{proof}[Proof of Theorem \ref{th:main3}]
Let $\varepsilon > 0$.
By contradiction, assume that there is a sequence $(g_k, \nu_k, \eta_k, \delta_k)$ such that
\begin{itemize}
\item[i)] $\nu_k$ is $g_k$-ergodic and converges to some $\nu$ in the weak-$*$ topology;
\item[ii)] $g_k \to f$ $\mathcal{C}^{r}$-weakly;
\item[iii)] $\eta_k, \delta_k \to 0$;
\item[iv)] for any $k$, $h_{g_k}(\nu_k) \geq h_{top}(g_k) - \delta_k$;
\item[v)] for any $k$, $\int_{\{ |g_k'| < \eta_k\}} -\log |g_k'| d \nu_k \geq \varepsilon$;
\item[vi)] $(h_{g_k}(\nu_k))_k$ converges.
\end{itemize}
Notice that for $0 < \eta < \eta' < 1$ and any $k$, we have
$$
\int_{\{ |g_k'| < \eta\}} -\log |g_k'| d \nu_k \leq \int_{\{ |g_k'| < \eta'\}} -\log |g_k'| d \nu_k.
$$
Hence, for any $\eta > 0$, since $\eta_k \to 0$, we have
$$
\limsup\limits_{k \to +\infty} \int_{\{ |g_k'| < \eta\}} -\log |g_k'| d \nu_k \geq \limsup\limits_{k \to +\infty} \int_{\{ |g_k'| < \eta_k\}} -\log |g_k'| d \nu_k.
$$
By letting $\eta \to 0$, we have that the $\gamma$ associated to the sequence $(g_k, \nu_k)$ satisfies
$$
\gamma \geq \limsup\limits_{k \to +\infty} \int_{\{ |g_k'| < \eta_k\}} -\log |g_k'| d \nu_k \geq \varepsilon.
$$
Then, by lower semi-continuity of the topological entropy for the $\mathcal{C}^0$-topology from \cite{Misiurewicz_LSC_entropy}, by Theorem \ref{th:UI}, by items iii) and iv) and by $\gamma \geq \varepsilon$, we have
\begin{align*}
h_{\rm{top}}(f) 
&\leq \lim\limits_{k \to +\infty} h_{g_k}(\nu_k) + \delta_k\\
&\leq (1 - \frac{\gamma}{\log ||f' ||_{\infty}}) (h_{\rm{top}}(f) - \frac{R(f)}{r}) + \frac{R(f)}{r}\\
&= (1 - \frac{\varepsilon}{\log ||f' ||_{\infty}})h_{\rm{top}}(f) + \frac{\varepsilon}{\log ||f' ||_{\infty}}\frac{R(f)}{r}.
\end{align*}

This gives
$
\frac{\varepsilon}{\log ||f' ||_{\infty}} h_{\rm{top}}(f) \leq \frac{\varepsilon}{\log ||f' ||_{\infty}}\frac{R(f)}{r}
$
which is a contradiction.
\end{proof}

\bibliographystyle{acm}

\newpage

\textsc{Alexandre Delplanque}, Sorbonne Universit\'e, Universit\'e Paris Cit\'e, CNRS, Laboratoire de Probabilit\'es, Statistique et Mod\'elisation

F-75005 Paris, France

\textit{E-mail:} 
{\fontfamily{lmtt}\selectfont
delplanque@lpsm.paris
}
\\

\textsc{Hengyi Li}, Institut de Mathématique d'Orsay, CNRS-UMR 8628,
Universit\'e Paris-Saclay, 91405 Orsay, France. 

\textit{E-mail:}
{\fontfamily{lmtt}\selectfont
hengyi.li@universite-paris-saclay.fr
}


\begin{thebibliography}{99}

\bibitem{BurguetExistMME}Burguet, D. Existence of measures of maximal entropy for $\mathcal{C}$ interval maps. {\em Proc. Amer. Math. Soc.}. \textbf{142}, 957-968 (2014)


\bibitem{BurguetJumpEntropy}Burguet, D. Jumps of entropy for $\mathcal{C}^r$ interval maps. {\em Fundamenta Mathematicae}. \textbf{231}, 299-317 (2015)


\bibitem{B}Burguet, D. Maximal Measure and Entropic continuity of Lyapunov Exponents for $\mathcal{C}^r$ Surface Diffeomorphisms with Large Entropy. {\em Ann. Henri Poincaré}. \textbf{25}, 1485-1510 (2024)


\bibitem{BCS}Buzzi, J., Crovisier, S. \& Sarig, O. Continuity properties of Lyapunov exponents for surface diffeomorphisms. {\em Inventiones Mathematicae}. (2022)


\bibitem{BuzziRuette}Buzzi, J. \& Ruette, S. Large entropy implies existence of a maximal entropy measure for interval maps. {\em Discrete And Continuous Dynamical Systems}. \textbf{14}, 673-688 (2006)

\bibitem{A}Delplanque, A. Hyperbolic absolutely continuous invariant measures for $\mathcal{C}^r$ one-dimensional maps. {\em ArXiv:2410.23021}. (2024)

\bibitem{P}Keane, M. \& Petersen, K. Easy and nearly simultaneous proofs of the ergodic theorem and maximal ergodic theorem. {\em Lecture Notes-Monograph Series}. pp. 248-251 (2006)

\bibitem{L}Li, H. Entropy Continuity of Lyapunov Exponents for Non-flat 1-dimensional Maps. {\em ArXiv:2412.19007}. (2024)



\bibitem{Misiurewicz1976}Misiurewicz, M. A Short Proof of the Variational Principle for a $\mathbb{Z}^n_+$ Action on a Compact Space. {\em Astérisque}. pp. 147-157 (1976)


\bibitem{Misiurewicz_LSC_entropy}Misiurewicz, M. Horseshoes for Continuous Mappings of an Interval. {\em Dynamical Systems}. pp. 125-135 (2011)



\bibitem{NewhouseContEntropy}Newhouse, S. Continuity Properties of Entropy. {\em Annals Of Mathematics}. \textbf{129}, 215-235 (1989)


\bibitem{SETE}Qian, M., Xie, J. \& Zhu, S. Smooth Ergodic Theory for Endomorphisms. (Springer Berlin Heidelberg,2009)


\bibitem{Ruelle1978}Ruelle, D. An inequality for the entropy of differentiable maps. {\em Boletim Da Sociedade Brasileira De Matemática - Bulletin/Brazilian Mathematical Society}. \textbf{9} pp. 83-87 (1978)

\bibitem{SarigNotes}Sarig, O. Lecture Notes on Ergodic Theory.  (2023)


\end{thebibliography}

\end{document}